\documentclass[11 pt] {amsart}
\usepackage[letterpaper]{geometry}
\usepackage{graphicx,tikz-cd,tikz}
\usepackage[pagewise]{lineno}
\usepackage{amsmath}
\usepackage{amssymb}
\usepackage{enumerate}
\usepackage{verbatim}
\usepackage{mathrsfs}
\usepackage{mathtools}
\usepackage{hyperref}
\usepackage{leftidx}
\usepackage{extpfeil}
\numberwithin{equation}{section}
\newcommand{\GL}{\operatorname{GL}}
\newcommand{\Sim}{\operatorname{Sim}}

\newcommand{\SO}{\operatorname{SO}}
\newcommand{\Sp}{\operatorname{Sp}}
\newcommand{\GSp}{\operatorname{GSp}}
\newcommand{\Spin}{\operatorname{Spin}}
\newcommand{\GSpin}{\operatorname{GSpin}}
\newcommand{\GSO}{\operatorname{GSO}}
\newcommand{\GO}{\operatorname{GO}}
\newcommand{\Gal}{\operatorname{Gal}}
\def\G{\ensuremath{\mathop{\textrm{\normalfont G}}}}

\def\Cent{\ensuremath{\mathop{\textrm{\normalfont Cent}}}}
\def\Im{\ensuremath{\mathop{\textrm{\normalfont Im}}}}

\theoremstyle{plain}
\newtheorem{theorem}{Theorem}[section]
\newtheorem{corollary}[theorem]{Corollary}
\newtheorem{proposition}[theorem]{Proposition}
\newtheorem{lemma}[theorem]{Lemma}
\newtheorem{remark}[theorem]{Remark}
\newtheorem{conjecture}[theorem]{Conjecture}

\usepackage{graphicx}

\DeclareMathOperator{\disc}{disc}

\DeclareMathOperator{\Span}{Span}
\DeclareMathOperator{\Ad}{Ad}
\DeclareMathOperator{\Rat}{Rat}
\DeclareMathOperator{\Hom}{Hom}
\DeclareMathOperator{\End}{End}
\DeclareMathOperator{\Res}{Res}
\usepackage{mathtools}
\usetikzlibrary{chains}
\usepackage{graphicx,tikz-cd}   \usepackage{amssymb,latexsym,amsmath,epsfig,amsthm,amsfonts,tikz-cd,dsfont}
\usepackage{booktabs}

\newcommand{\A}{\mathbb{A}}
\newcommand{\C}{\mathbb{C}}

\newcommand{\qand}{\quad\text{and}\quad}
\tikzset{
  symbol/.style={
    draw=none,
   every to/.append style={
      edge node={node [sloped, allow upside down, auto=false]{$#1$}}}
  }
  }

\date{}

\author{Melissa Emory }
\address{Department of Mathematics, University of Toronto, Toronto, Canada}
\email{memory@math.toronto.edu}

\title[On the global GGP Conjecture for GSpin]{ On the global Gan-Gross-Prasad Conjecture\\ for General Spin Groups}

\begin{document}

\maketitle

\begin{abstract}
We formulate a global Gan-Gross-Prasad conjecture for general spin groups.  That is, we formulate a conjecture on a relation between periods of certain automorphic forms on $\GSpin_{n+1} \times \GSpin_n$ along the diagonal subgroup $\GSpin_n$ and some $L$-values.  To support the conjecture, we show that the conjecture holds for $n=2$ and $3$ and for certain cases for $n=4$.
\end{abstract}

\section{Introduction}
In 1992 Gross  and  Prasad (\cite{GP92}) conjectured that the non-vanishing of periods of automorphic forms on $\SO_{n+1} \times \SO_{n}$ along the diagonal subgroup $\SO_n$ is equivalent to the non-vanishing of certain automorphic $L$-functions at the central critical value.  Gan, Gross and Prasad extended this conjecture, now known as the {\sl Gan-Gross-Prasad (GGP) conjecture}, to the remaining classical groups (\cite {GGP12}). The original Gross-Prasad conjecture was refined by Ichino  and  Ikeda in \cite{II10}, where an explicit relationship was conjectured between the period integral and the central critical $L$-values. An analogous conjecture was developed for unitary groups  by N. Harris (\cite{Har14}). The purpose of this paper is to formulate a similar conjecture for a non-classical group known as the general spin group ($\GSpin$), and to verify the conjecture for the first three cases essentially by interpreting the following known results: the Waldspurger formula (\cite{Wal85}) for $n=2$, Ichino's triple product formula (\cite{Ich08}) for $n=3$, and a result of Gan-Ichino (\cite{GI11}) for $n=4$.

Let us first recall the original global Gross-Prasad conjecture. Let $F$ be a number field and $\A$ the ring of adeles over $F$. Let $(V_n, q_n)\subset (V_{n+1}, q_{n+1})$ be an inclusion of quadratic spaces of respective dimensions $n$ and $n+1$ over $F$, so that $q_{n+1}|_{V_n}=q_n$, where we assume $n$ is at least two and $V_n$ is not isomorphic to the hyperbolic plane. Then we have the natural inclusion $\SO_n \subset \SO_{n+1}$ of the corresponding special orthogonal groups $\SO_n:=\SO(V_n)$ and $\SO_{n+1}:=\SO(V_{n+1})$ over $F$, which gives rise to the inclusion $\SO_n(\A) \subset \SO_{n+1}(\A)$. Let $\pi_n$ and $\pi_{n+1}$ be irreducible tempered cuspidal automorphic representations of $\SO_n(\A)$ and $\SO_{n+1}(\A)$, respectively. The original global Gross-Prasad conjecture is as follows.

\begin{conjecture}[Original Global Gross-Prasad  Conjecture \cite{GP92}]
Assume that for every place $v$ of $F$, $\Hom_{\SO_n(F_v)}(\pi_{{n+1},v} \otimes \pi_{n,v}, \mathbb{C})\neq \{0\}$. Then there exist vectors $\phi \in V_{\pi_{n+1}}$ and $f\in V_{\pi_n}$ such that
\[
\int_{\SO_n(F)\backslash \SO_n(\A)}\phi(g)f(g)dg \neq 0
\]
if and only if the tensor product $L$-function $L(1/2,\pi_{n+1} \times \pi_n)$ does not vanish.
\end{conjecture}

Ichino-Ikeda (\cite{II10}) refined this conjecture by writing down an explicit (conjectural) relationship between the period integral and $L(1/2,\pi_{n+1} \times\pi_n)$ as follows. First define
\[
\mathcal{P}:V_{\pi_{n+1}}\otimes V_{\pi_n}\longrightarrow\C
\]
by
\[
\mathcal{P}(\phi,f)=\int_{\SO_n(F)\backslash \SO_n(\A)}\phi(g)f(g)dg,
\]
for $\phi\in V_{\pi_{n+1}}$ and $f\in V_{\pi_n}$, where $dg$ is the Tamagawa measure of $\SO_n(\A)$. This is, of course, nothing but the period integral of the above original Gross-Prasad conjecture. The basic idea of Ichino-Ikeda is to define a ``local period" $\alpha_v$ by using the matrix coefficients of the local representations $\pi_{n+1, v}$ and $\pi_{n, v}$ so that the infinite product $\prod_v\alpha_v$ is defined. They then conjecture that the global $|\mathcal{P}(\phi,f)|^2$ is proportional to the product $\prod_v\alpha_v(\phi_v,f_v)$ for factorizable $\phi=\otimes_v\phi_v$ and $f=\otimes_vf_v$ and the $L$-value $L(1/2,\pi_{n+1} \times \pi_n)$ appears in the constant of proportionality.

To state their conjecture more precisely, first let
\[
\mathcal{B}_{{\pi}_{n+1}}:V_{\pi_{n+1}} \otimes \overline{V}_{\pi_{n+1}} \rightarrow \mathbb{C}\qand
\mathcal{B}_{{\pi}_{n}}:V_{\pi_{n}} \otimes \overline{V}_{\pi_n} \rightarrow \mathbb{C}
\]
be the Petersson pairings defined as usual via the Tamagawa measures. Then fix isomorphisms
\[
\pi_n\cong \otimes_v \pi_{n,v}\qand\pi_{n+1}\cong\otimes_v\pi_{n+1,v},
\]
and decompositions
\[
\mathcal{B}_{\pi_{n+1}}=\prod_v \mathcal{B}_{\pi_{n+1,v}} \qand \mathcal{B}_{\pi_{n}}=\prod_v \mathcal{B}_{\pi_{n,v}},
\]
where
\[
\mathcal{B}_{\pi_{n+1,v}}: \pi_{n+1,v} \otimes \overline{\pi}_{n+1,v} \rightarrow \mathbb{C}\qand\mathcal{B}_{\pi_{n,v}}: \pi_{n,v} \otimes \overline{\pi}_{n,v} \rightarrow \mathbb{C}
\]
are local pairings. Also fix a decomposition $dg=\prod_vdg_v$ of the Tamagawa measure $dg$ on $\SO_n(\A)$.

Then define an $\SO_n(F_v)\times\SO_n(F_v)$-invariant functional
\begin{equation}\label{E:local_integral_alpha}
\alpha^{\natural}_v:(\pi_{n+1,v}\boxtimes\overline{\pi}_{n+1,v})\otimes(\pi_{n,v}\boxtimes\overline{\pi}_{n,v})\rightarrow \mathbb{C}
\end{equation}
by
\begin{align*}
&\alpha^{\natural}_v(\phi_{1,v},\phi_{2,v};f_{1,v},f_{2,v}):=\\
&\hspace{1in}\int_{\SO_n(F_v)}\mathcal{B}_{\pi_{n+1,v}}(\pi_{n+1,v}(g_{v})\phi_{1,v},\phi_{2,v})
\mathcal{B}_{\pi_{n,v}}(\pi_{n,v}(g_{v})f_{1,v},f_{2,v})dg_{v}
\end{align*}
for $\phi_{1,v},\phi_{2,v}$ in ${\pi_{n+1,v}}$ and $f_{1,v},f_{2,v}$ in $\pi_{n,v}$. Ichino-Ikeda have proven that if $\pi_{i,v}$ is tempered then the integral for $\alpha^{\natural}_v$ converges absolutely, and
\[
\alpha^{\natural}_v(\phi_{1, v},\phi_{2, v}; f_{1, v}, f_{2, v})=\Delta_{\SO_{n+1,v}}
\frac{L_v(1/2,\pi_{n,v}\times\pi_{n+1,v})}{L_v(1,\pi_{n,v},\Ad)L_v(1,\pi_{n+1,v},\Ad)}
\]
for almost all $v$, where
\[
\Delta_{\SO_{n+1},v}:=\begin{cases} \zeta_v(2)\zeta_v(4)\cdots\zeta_v(2m) & \text{ if } \dim V_{n+1,v} = 2m+1,\\
\zeta_v(2)\zeta_v(4)\cdots \zeta_v(2m-2) \cdot L_v(m, \chi_{V_{n+1, v}}) & \text { if } \dim V_{n+1,v}=2m,
\end{cases}
\]
where $\chi_{V_{n+1, v}}$ is the character associated with the discriminant of the quadratic form associated to $V_{n+1, v}$.  Accordingly, for all $v$ define the normalized $\SO_{n+1}(F_v)\times\SO_n(F_v)$-invariant functional
\[
\alpha_v:(\pi_{n+1,v}\boxtimes\overline{\pi}_{n+1,v})\otimes(\pi_{n,v}\boxtimes\overline{\pi}_{n,v})\rightarrow \mathbb{C}
\]
by setting
\[
\alpha_v
=\Delta_{\SO_{n+1}, v}^{-1}\dfrac{L_v(1,\pi_{n,v},\Ad)L_v(1,\pi_{n+1,v},\Ad)}{L_v(1/2,\pi_{n,v}\times\pi_{n+1,v})}
\alpha^{\natural}_v,
\]
so that the infinite product
\[
\prod_v\alpha_v(\phi_{1,v},\phi_{2,v};f_{1,v},f_{2,v})
\]
is well-defined. Also we write
\[
\alpha_v(\phi_v, f_v):=\alpha_v(\phi_{v},\phi_{v};f_{v},f_{v})
\]
for $\phi_v\in\pi_{n+1, v}$ and $f_v\in\pi_{n, v}$.

Using these notations, we can state the Ichino-Ikeda refinement of the global Gross-Prasad conjecture for the special orthogonal groups as follows.
\begin{conjecture}[Ichino-Ikeda Refinement]
Assume $\pi_{n+1}$ and $\pi_n$ are tempered cuspidal automorphic representations of $\SO_{n+1}(\A)$ and $\SO_n(\A)$, respectively, and $\pi_{n+1}$ and $\pi_n$ appear with multiplicity one in the discrete spectrum.  Then for each factorizable $\phi=\otimes_v\phi_v\in V_{\pi_{n+1}}$ and $f=\otimes_vf_v\in V_{\pi_n}$, we have
\[
|\mathcal{P}(\phi,f)|^2=\frac{\Delta_{\SO_{n+1}}}{2^{\beta}}\frac{L(1/2,\pi_n \times \pi_{n+1})}{L(1,\pi_n,\Ad)L(1,\pi_{n+1},\Ad)}\prod_v \alpha_v(\phi_v,f_v),
\]
where $2^{\beta}$ is the product of cardinalities of the component groups attached to the $L$-packets for $\pi_{n+1}$ and $\pi_n$ and
\[
\Delta_{\SO_{n+1}}:=\begin{cases} \zeta(2)\zeta(4)\cdots\zeta(2m) & \text{ if } \dim V_{n+1} = 2m+1,\\
\zeta(2)\zeta(4)\cdots \zeta(2m-2) \cdot L(m, \chi_{V_{n+1}}) & \text { if } \dim V_{n+1}=2m,
\end{cases}
\]
where $\chi_{V_{n+1}}$ is the quadratic character associated with $V_{n+1}$.
\end{conjecture}

\begin{remark}
It should be noted that what is denoted by $\mathcal{P}(\phi, f)$ in \cite{II10} is our $|\mathcal{P}(\phi,f)|^2$.  See \cite[Conjecture 6.2.1]{Xue} for a similar conjecture if $n$ is odd and $\pi_{n+1}$ appears with multiplicity two in the discrete spectrum.
\end{remark}

This refined conjecture for $n=2$ follows from the well-known Waldspurger formula (\cite{Wal85}) and the one for $n=3$ follows from Ichino's triple product formula (\cite{Ich08}). Also for $n=4$, Gan and Ichino in \cite{GI11} have proven the conjecture under certain assumptions for $\pi_{n+1}$ and $\pi_n$.

\quad

Our goal in this paper is to generalize this conjecture for the general spin groups and verify it for the cases $n=2, 3$ and $4$ by using \cite{Wal85}, \cite{Ich08} and \cite{GI11}, respectively, as above.

Let us first briefly recall some generalities of the general spin group. Let $(V_n, q_n)$ be a quadratic space over $F$ of dimension $n$. The general spin group associated with $(V_n, q_n)$, which we denote by $\GSpin(V_n)$ or simply by $\GSpin_n$, is a reductive group over $F$ such that we have the short exact sequence
\[
1\longrightarrow \GL_1\longrightarrow\GSpin(V_n)\longrightarrow\SO(V_n)\longrightarrow 1.
\]
It should be noted that $\GL_1$ is in the center of $\GSpin(V_n)$ and is the connected component $Z_n^\circ$ of the center if $n> 2$.  If $n=2$ then $\GSpin_2$ is commutative and hence the connected component of the center is larger than this $\GL_1$. However, as a convention in this paper, we set
\[
Z_n^\circ=\GL_1
\]
even when $n=2$. Also the group $\GSpin(V_n)$ is equipped with a homomorphism
\[
N:\GSpin(V_n)\longrightarrow\GL_1,
\]
which is called the spinor norm. Note that for $z\in \GL_1\subseteq Z_n^\circ$ we have $N(z)=z^2$.

Next assume we have an inclusion $(V_n, q_n)\subseteq (V_{n+1}, q_{n+1})$ of quadratic spaces. Then we have the natural inclusion $\GSpin(V_n)\subseteq \GSpin(V_{n+1})$, which makes the diagram
\[
\begin{tikzcd}
1\ar[r]&Z_{n+1}^\circ\ar[r]&\GSpin(V_{n+1})\ar[r]&\SO(V_{n+1})\ar[r]&1\\
1\ar[r]&Z_{n}^\circ\ar[r]\ar[u, symbol={=}]&\GSpin(V_{n})\ar[r]\ar[u, symbol=\subseteq]&\SO(V_n)\ar[u, symbol=\subseteq]\ar[r]&1
\end{tikzcd}
\]
commute.

Now, let $\pi_{n+1}$ and $\pi_n$ be tempered cuspidal automorphic representations of $\GSpin_{n+1}(\A)$ and $\GSpin_n(\A)$, respectively, and let $\omega_{\pi_{n+1}}$ and $\omega_{\pi_n}$ be the restrictions to $Z_n^\circ(\A)=\A^\times$ of the central characters of $\pi_{n+1}$ and $\pi_n$, respectively, so that $\omega_{\pi_{n+1}}$ and $\omega_{\pi_n}$ are Hecke characters.  Furthermore, assume that the product $\omega_{\pi_{n+1}}\omega_{\pi_n}$ has a square root; namely there exists a Hecke character $\omega:\A^\times\to\C^1$ such that
\[
\omega^2=\omega_{\pi_{n+1}}{\omega_{\pi_{n}}}.
\]
Note that such $\omega$ is not unique and two such $\omega$'s differ by a quadratic character. For each such $\omega$, we define the global period
\[
\mathcal{P}_{\omega}:V_{\pi_{n+1}}\otimes V_{\pi_n}\longrightarrow\C
\]
by
\[
\mathcal{P}_\omega(\phi,f):=\int\limits_{Z_{n}^\circ(\A)\GSpin_n(F)\backslash \GSpin_n(\A)}\phi(g)f(g)\omega(N(g))^{-1}dg,
\]
where $\phi\in V_{\pi_{n+1}}$ and $f\in V_{\pi_n}$, and $dg$ is the Tamagawa measure of $\GSpin_n(\A)$. Because of the assumption on the central characters that $\omega^2=\omega_{\pi_{n+1}}\omega_{\pi_n}$, this integral is well-defined.  Then in the same way as the $\SO$-case, we define the local period $\alpha_{\omega_v}(\phi_v,f_v)$ in such a way that $\alpha_{\omega_v}(\phi_v,f_v)=1$ for almost all $v$ so that the product $\prod_v\alpha_{\omega_v}(\phi_v,f_v)$ makes sense, and we make an analogous conjecture.

To be precise, by fixing isomorphisms
\[
\pi_n\cong \otimes_v \pi_{n,v}\qand\pi_{n+1}\cong\otimes_v\pi_{n+1,v}
\]
and decompositions
\[
\mathcal{B}_{\pi_{n+1}}=\prod_v \mathcal{B}_{\pi_{n+1,v}} \quad \text {and} \quad \mathcal{B}_{\pi_{n}}=\prod_v \mathcal{B}_{\pi_{n,v}},
\]
where the global $\mathcal{B}_{\pi_{n+1}}$ and $\mathcal{B}_{\pi_{n}}$ are the Petersson pairings defined by the Tamagawa measures, we define the $\GSpin_{n+1}(F_v)\times\GSpin_n(F_v)$-invariant functional
\[
\alpha^{\natural}_{\omega_v}:(\pi_{n+1,v}\boxtimes\overline{\pi}_{n+1,v})\otimes(\pi_{n,v}\boxtimes\overline{\pi}_{n,v})
\rightarrow\C
\]
by
\begin{align*}
&\alpha^{\natural}_{\omega_v}(\phi_{1,v},\phi_{2,v};f_{1,v},f_{2,v})\\
&:=\int\limits_{Z_n^\circ(F_v)\backslash\GSpin_n(F_v)}\mathcal{B}_{\pi_{n+1,v}}(\pi_{n+1,v}(g_{v})\phi_{1,v},\phi_{2,v})
\mathcal{B}_{\pi_{n,v}}(\pi_{n,v}(g_{v})f_{1,v},f_{2,v})\omega_v(N(g_v))^{-1}dg_{v}
\end{align*}
for $\phi_{1,v},\phi_{2,v}$ in ${\pi_{n+1,v}}$ and $f_{1,v},f_{2,v}$ in $\pi_{n,v}$, where we also fix the factorization $dg=\prod_vdg_v$ of the Tamagawa measure on $\GSpin_n(\A)$.

Then we prove that if $\pi_{n+1, v}$ and $\pi_{n, v}$ are tempered then the integral for $\alpha^{\natural}_{\omega_v}(\phi_v,f_v)$ converges absolutely, and
\[
\alpha^{\natural}_{\omega_v}(\phi_{1, v},\phi_{2, v}; f_{1, v}, f_{2, v})=\Delta_{\SO_{n+1,v}}
\frac{L_v(1/2,\pi_{n,v}\times\pi_{n+1,v}\otimes\omega_v^{-1})}{L_v(1,\pi_{n,v},\Ad)L_v(1,\pi_{n+1,v},\Ad)}
\]
for almost all $v$, where $\Delta_{\SO_{n+1,v}}$ as before. Hence if we normalize $\alpha^{\natural}_{\omega_v}$ by setting
\[
\alpha_{\omega_v}
:=\Delta_{\SO_{n+1},v}^{-1}\frac{L_v(1,\pi_{n,v},\Ad)L_v(1,\pi_{n+1,v},\Ad)}{L_v(1/2,\pi_{n,v}\times\pi_{n+1,v}\otimes\omega_v^{-1})}
\alpha^{\natural}_{\omega_v},
\]
then the infinite product
\[
\prod_v\alpha_{\omega_v}(\phi_{1,v},\phi_{2,v};f_{1,v},f_{2,v})
\]
is well-defined. Also we write
\[
\alpha_{\omega_v}(\phi_v, f_v):=\alpha_{\omega_v}(\phi_{v},\phi_{v};f_{v},f_{v})
\]
for $\phi_v\in\pi_{n+1, v}$ and $f_v\in\pi_{n, v}$.

Then we make the following conjecture, which we call the global GGP conjecture for $\GSpin$.

\begin{conjecture} [The global GGP conjecture for $\GSpin$]\label{C:conj}
Let $\pi_n \cong \otimes_v \pi_{n,v}$ and $\pi_{n+1} \cong \otimes_v \pi_{n+1,v}$ be irreducible tempered cuspidal automorphic representations of $G_n(\mathbb{A})$ and $G_{n+1}(\mathbb{A})$, respectively, and we assume $\pi_{n+1}$ and $\pi_n$ appear with multiplicity one in the discrete spectrum. Assume there exists $\omega$ such that $\omega^2=\omega_{\pi_{n+1}}\omega_{\pi_n}$.  Then for each factorizable $\phi=\otimes_v\phi_v\in V_{\pi_{n+1}}$ and $f=\otimes_vf_v\in V_{\pi_n}$, we have
\[
\left|\mathcal{P}_\omega(\phi, f)\right|^2 =\frac{\Delta_{\SO_{n+1}}}{2^{\beta}}\frac{L(1/2, \pi_{n+1}  \times \pi_{n} \otimes \omega^{-1})}{L(1, \pi_{n+1} , \Ad)L(1, \pi_{n} , \Ad)}\prod_v \alpha_{\omega_v}(\phi_v, f_v),
\]
where $\Delta_{\SO_{n+1}}$ is as before.
\end{conjecture}

Let us mention that the $L$-function $L(\pi_{n+1} \times \pi_n \otimes \omega^{-1})$ is conjecturally holomorphic at $s=1/2$ by, say, \cite[Theorem 5.1]{PSR87}. Also the adjoint $L$-functions $L(s,\pi_{n+1},\Ad)$ and $L(1,\pi_n,\Ad)$ are conjecturally nonzero holomorphic at $s=1$ as in \cite[(3.2) and pg. 483]{LM15}.

An interesting quantity in the above conjecture is $2^\beta$, which is conjecturally related to the cardinalities of the component groups attached to the $L$-parameters of $\pi_{n+1}$ and $\pi_n$. To discuss this issue, we first set up some general notations. Let $G$ be a reductive group over our number field $F$ and let $\pi$ be a cuspidal automorphic representation of $G(\A)$. Furthermore, let $\mathcal{L}_F$ be the hypothetical global Langlands group of $F$ and let $\phi=\phi_{\pi}: \mathcal{L}_F \rightarrow \leftidx{^}LG = \widehat{G} \rtimes \Gamma_F$ be the hypothetical global Langlands parameter of $\pi$, where $\Gamma_F=\Gal(\overline{F}/F)$. Set $S_{\phi}:=\Cent(\Im(\phi),\widehat{G})$ and define
\[
\mathcal{S}_{\phi}:=S_{\phi}/S_{\phi}^0Z(\widehat{G})^{\Gamma_F},
\]
where $S^0_{\phi}$ is the identity component of the complex reductive group $S_{\phi}$, $Z(\widehat{G})$ is the center of $\widehat {G}$, and $Z(\widehat{G})^{\Gamma_F}$ is the subgroup of invariants in $Z(\widehat{G})$ under the natural action of $\Gamma_F$.  Denote by $\widehat{G}_{sc}$ the simply connected cover of the derived group $\widehat{G}_{der}$ of $\widehat{G}$, and by ${S}_{\phi,sc}$ the full pre-image of ${S}_{\phi}$ in $\widehat{G}_{sc}$. We then define
\[
\mathcal{S}_{\phi,sc}:=S_{\phi,sc}/S_{\phi,sc}^{\circ}.
\]
By using these notations, we make the following conjecture.
\begin{conjecture}\label{C:conj2}
Let $\phi_{n+1}$ and $\phi_n$ be the (conjectural) global Langlands parameters of $\pi_{n+1}$ and $\pi_n$, respectively. If $\pi_{n+1}$ and $\pi_n$ appear with multiplicity one in the discrete spectrum, then 
$$2^{\beta}=4|\mathcal{S}_{\phi_n}||\mathcal{S}_{\phi_{n+1}}|=\dfrac{1}{2}|\mathcal{S}_{\phi_n,sc}||\mathcal{S}_{\phi_{n+1},sc}|.$$
\end{conjecture}
\begin{remark}
If $\pi_{n+1}$ or $\pi_n$ appear with multiplicity greater than one in the discrete spectrum, then Conjecture \ref{C:conj2} should be modified in a similar way as for special orthogonal groups (see Conjecture 6.2.1 in \cite{Xue}). 
\end{remark}

\quad

As the last thing in this introduction, let us mention that if the central characters of $\pi_{n+1}$ and $\pi_n$ are both trivial, so that $\omega_{\pi_{n+1}}=\omega_{\pi_n}=1$, then $\pi_{n+1}$ and $\pi_n$ can be seen as automorphic representations of $\SO_{n+1}(\A)$ and $\SO_n(\A)$, respectively. In this case, if one chooses $\omega=1$, one can readily see that our conjecture reduces to that of Ichino-Ikeda. Hence our conjecture should be considered as a ``generalization" of the Ichino-Ikeda conjecture rather than an analogue of it.

\quad

This paper is organized as follows.  In \S ~2, we review the general theory of GSpin and discuss Conjecture \ref{C:conj2}.  In \S ~3, we establish the convergence of the integral and then compute the integral for unramified data.  In \S ~4, we wrap-up our formulation of the conjecture, and then establish the conjecture for the $n=2,3$ and $4$ cases.

\quad

\noindent{\bf Notations}: If $\pi$ is a representation of a group $G$, we denote the space of $\pi$ by $V_{\pi}$. If $\pi$ admits a central character we write $\omega_\pi$ for the central character of $\pi$ restricted to the connected component of the center of $G$. Assume that the space $V_{\pi}$ is a space of functions or maps on the group $G$ and $\pi$ is a representation of G on $V_{\pi}$ defined by right translation (for example, when $\pi$ is an automorphic sub-representation). Let $H$ be a subgroup of $G$.  We define $\pi \|_H$ to be the representation of $H$ realized in the space
$$V_{\pi \|_H}:=\{f|_H: f\in V_{\pi}\}$$
of restrictions of $f \in V_{\pi}$ to $H$ on which $H$ acts by right translation.  Namely $\pi \|_H$ is the representation obtained by restricting the functions in $V_{\pi}$.

For a reductive group $G$ over $F$, we usually identify the group $G$ with its $F$-rational points $G(F)$. We denote by $\Gamma_F$ the absolute Galois group of $F$. If $(V, q)$ is a quadratic space over $F$, we denote its discriminant by $\disc(V)$, which is always viewed in $F^\times\slash F^{\times 2}$. We denote by $\mathbb{H}$ the hyperbolic plane over $F$, namely the unique 2-dimensional split quadratic space.

\quad

\noindent{\bf Acknowledgements}: This work stems from the author's PhD thesis. The author thanks her advisor Shuichiro Takeda for suggesting this problem, and his helpful advice and support.   The author also thanks Atsushi Ichino for helpful discussions and for directing the author to the work by Xue in \cite{Xue} and Lapid-Mao in \cite{LM15}; as well as Hang Xue for discussions regarding Conjecture 1.5 and Remark 1.6.  \\

\section{The General Spin Group} In this section, we begin by defining the general spin group over any field $F$ and  then explicitly compute the first five cases.  We then let $F$ be a number field, and define the local and global $L$-functions for $\GSpin$.  The section is concluded with the component groups for the Langlands parameters for $\GSpin$ and we discuss Conjecture \ref{C:conj2}.

\subsection{The group $\GSpin$}
In the past literature such as \cite{Asg00, AS06, AC16}, the quasi-split general spin group is often defined in terms of roots and coroots, which is useful, for example, when one computes the dual group. In this paper, however, we give an alternate definition in terms of the Clifford algebra, which can be done even for the non-quasi-split case, and which more naturally gives the inclusion $\GSpin_n\subseteq\GSpin_{n+1}$.

Let $(V,q)$ be a quadratic space with quadratic form $q$ over an arbitrary field $F$ of characteristic different from 2 with $\dim V=n$. (Of course we are interested in the case when $F$ is local or global, but in this subsection $F$ can be arbitrary.) Let $T(V)$ denote the tensor algebra of $V$, that is,
\[
T(V) := \bigoplus_{k=0}^{\infty}V^{\otimes k} =  F \oplus V \oplus (V \otimes V) \oplus (V \otimes V \otimes V) \oplus \cdots.
\]
Let $I(q)$ be the two sided ideal of $T(V)$ generated by the elements of the form
\[
v\otimes v - q(v) \cdot 1
\]
where $v\in V$, and define
\[
C(V):=C(V, q)=T(V)/I(q),
\]
which is called the Clifford algebra associated with $(V, q)$. We write $v_1 \otimes v_2 \otimes  \cdots \otimes v_k = v_1v_2 \cdots v_k$ and in particular  $v^2=q(v)$.

\begin{lemma}\label{L:orth}
Let $x$ and $y$ be orthogonal in $V$; namely $B_q(x, y)=0$, where $B_q$ is the bilinear form associated with $q$.  Then in $C(V)$,
$$x\cdot y = - y \cdot x.$$
\end{lemma}
\begin{proof}
By definition of $B_q$, $B_q(x,y)=0$ implies $q(x+y)-q(x)-q(y) = 0 $, namely $q(x+y)=q(x)+q(y)$.  Thus, we have
$$ (x+y)(x+y)= x \cdot x + y \cdot y. $$ The lemma follows.
\end{proof}

Let $\{x_1,x_2,\dots,x_n\}$ be an orthogonal basis of $V$.  Then each element in $C(V)$ is a linear combination of the elements of the form $x_{i_1}x_{i_2} \cdots x_{i_k}$, where $x_{i_j}\in\{x_1,x_2,\dots,x_n\}$. Furthermore, by the above lemma along with $v^2=q(v)\in F$, we can readily see that each element in $C(V)$ is a linear combination of the elements of the form $x_{i_1} \cdots x_{i_k}$ with $i_1 < i_2 < \cdots <i_k$. We then define
$$C^k(V) := \left\{ \displaystyle\sum_i c_ix_{i_1} \cdots x_{i_k} \;:\; i_1 < i_2 < \cdots <i_k\right\},$$
so that we have
$$C(V)= C^0(V) \oplus C^1(V) \oplus \cdots \oplus C^n(V).$$ Further, we define
$$C^+(V) = C^0(V) \oplus C^2(V) \oplus C^4(V) \oplus \cdots $$
and $$C^-(V) = C^1(V) \oplus C^3(V) \oplus C^5(V) \oplus \cdots ,$$
and call them the even Clifford algebra and the odd Clifford algebra, respectively, so that we have
$$C(V) = C^+(V) \oplus C^{-}(V).$$
Let us note that
\[
\dim C^{+}(V)=\dim C^{-}(V)=2^{n-1}
\]
and so $\dim C(V)=2^n$ (see, for example, \cite[Theorem 1.8 and Corollary 1.9]{Lam05}.)

With this said, the general spin group $\GSpin(V)$ associated with $(V, q)$ is defined as
\[
\GSpin(V) := \{g \in C^+(V) : g^{-1} \text{ exists and }  gVg^{-1} = V\}
\]
(see, for example, \cite[Section 3.2]{Del72}.)
We sometimes write $\GSpin_n=\GSpin(V)$ when $V$ is clear from the context.

The Clifford algebra $C(V)$ is equipped with a natural involution (which we call the canonical involution on $C(V)$) defined by
$$(v_1\dots  v_k)^* = v_k  \dots  v_1,$$
where $v_i\in V$, giving rise to the map
\[
N:C(V)\longrightarrow C(V),\quad N(x)=xx^*,
\]
for $x\in C(V)$. It is immediate that $C^+(V)$ is closed under the canonical involution thanks to Lemma \ref{L:orth}. Now if $g\in\GSpin(V)$ then we have $N(g)\in C^0(V)^\times=F^\times$, say, by \cite[Lemma 3.2, pg. 335]{Sch85}, and we obtain the group homomorphism
\[
N:\GSpin(V)\longrightarrow F^\times,
\]
which we call the spinor norm on $\GSpin(V)$.

\begin{theorem}\label{embed}
Let $(V_n, q_n) \subset (V_{n+1}, q_{n+1})$ be an inclusion of quadratic spaces, so that $q_{n+1}|_{V_n}=q_n$  Then we have the natural inclusion $\GSpin(V_n) \subset \GSpin(V_{n+1})$.
\end{theorem}
\begin{proof}
If $V_n\subset V_{n+1}$ is an inclusion of quadratic spaces then $C^+(V_n)\subset C^+(V_{n+1})$. Now suppose $g \in C^+(V_n)$ is  such that $gV_ng^{-1} \subseteq V_n$.  Then it suffices to show that $gV_{n+1}g^{-1} \subseteq V_{n+1}$.   We can choose an orthogonal basis such that
$$V_{n+1}=\Span \{v_1,v_2,\dots,v_n,v_{n+1}\}$$ and such that $\{v_1,v_2,\dots,v_n\}$ is an orthogonal basis for $V_n$.  Then we can write $$g=\displaystyle\sum_i \alpha_i v_{i_1} \dots v_{i_k},$$ where $i_k \leq n$ and $\alpha_i\in F^{\times}$.   Since $g$ is in the even Clifford algebra there are an even number of vectors appearing and by Lemma \ref{L:orth} $gv_{n+1}=v_{n+1}g$, and the claim follows.
\end{proof}

There is a natural homomorphism $p:\GSpin(V)\to \SO(V)$ sending $g \in \GSpin(V)$ to the map $v \mapsto gvg^{-1}$, giving the short exact sequence of algebraic groups
\begin{equation}\label{sesgl}
\begin{tikzcd}
1 \arrow[r] & \GL_1 \arrow[r] & \GSpin(V) \arrow[r,"p"]  & \SO(V) \arrow[r] & 1,
\end{tikzcd}
\end{equation}
where $\ker p=\GL_1=C^0(V)^\times$ (see, for example \cite[Section 3]{Del72}). Hence if $(V, q) \subset (V', q')$ is the inclusion of quadratic spaces as in the above corollary, we have the commutative diagram
\[
\begin{tikzcd}
1\ar[r]&\GL_1\ar[r]&\GSpin(V')\ar[r]&\SO(V')\ar[r]&1\\
1\ar[r]&\GL_1\ar[r]\ar[u, symbol={=}]&\GSpin(V)\ar[r]\ar[u, symbol=\subseteq]&\SO(V)\ar[u, symbol=\subseteq]\ar[r]&1\rlap{\, ,}
\end{tikzcd}
\]
because we have the obvious equality $C^0(V)=C^0(V')$.

The following should be mentioned.
\begin{lemma} \label{center}
Assume $\dim V=n>2$ and let $Z^{\circ}$ be the identity component of the center of $\GSpin(V)$. Then $\ker p=Z^{\circ}$, where $p:\GSpin(V)\to\SO(V)$ is as above, and hence $Z^\circ=\GL_1$.
\end{lemma}
\begin{proof}
First it is clear that $\ker p\subseteq Z^\circ$ because $\ker p=\GL_1$ is in the center of $\GSpin(V)$ and $\GL_1$ is connected. So it suffices to show $Z^\circ\subseteq\ker p$. So let $a\in Z^\circ$. Then $p(a)$ is in the center $Z_{\SO}$ of $\SO(V)$. Now if $n$ is odd, then $Z_{\SO}=1$ and hence $p(a)=1$. If $n$ is even, then $Z_{\SO}=\{\pm 1\}$. But since $\{\pm 1\}$ is disconnected and $Z^{\circ}$ is connected (as an algebraic group), we also have $p(a)=1$. So in either case, we have $a\in p^{-1}(1)=\ker p$.
\end{proof}

Let us note that if $\dim V\leq 2$ then it is well-known that the entire group $\SO(V)$ is commutative. Similarly, as we will see in the next subsection, the general spin group $\GSpin(V)$ is also commutative.

\subsection{Low rank $\GSpin$}\label{iso}
In this subsection, we will explicitly compute $\GSpin(V_n)$ when $n=\dim V_n$ is small. This is done by using that $\GSpin(V_n)$ is a subgroup of the group of similitudes
$$ \Sim_n:=\{g \in C^+(V_n):N(g) \in C^{0}(V)^{\times}=F^{\times}\}$$ (see, for example, \cite[Prop. 13.10]{KMRM98}) along with the following result.

\begin{theorem}\label{CSA}
Let $d=\disc(V_n)$ and $E=F(\sqrt{d})$. Then we have
\[
C^+(V_n)=
\begin{cases}
A&\text{if $n$ is odd};\\
A\times A&\text{if $n$ is even and $d=1$};\\
A_E&\text{if $n$ is even and $d\neq 1$},
\end{cases}
\]
where $A$ is a central simple algebra over $F$ and $A_E$ is a central simple algebra over $E$. Note that since $\dim_FC^+(V_n)=2^{n-1}$ we must have $\dim_FA=2^{n-1}$ if $n$ is odd, and $\dim_FA=\dim_E A_E=2^{n-2}$ if $n$ is even.

Also
\[
\text{the involution $*$ is}
\begin{cases}
\text{unitary}&\text{if $n\equiv 2, 6\mod{8}$};\\
\text{symplectic}&\text{if $n\equiv 3, 4, 5\mod{8}$};\\
\text{orthogonal}&\text{if $n\equiv 0, 1, 7\mod{8}$},
\end{cases}
\]
and furthermore if $n\equiv 0, 4\mod{8}$ and $C^{+}(V_n) = A \times A$ then * is of orthogonal or symplectic type on each factor of $C^+(V_n).$ 
\end{theorem}
\begin{proof}
See \cite[Theorem 2.4,2.5] {Lam05} and \cite[(8.4) Proposition]{KMRM98}.
\end{proof}

Though known to the experts, using that $\GSpin_n$ is a connected subgroup of $\Sim_n$ we can easily compute $\GSpin_n=\GSpin(V_n)$ for $n=1, 2, 3, 4$ and the split case of $\GSpin_5$ by showing for these low rank cases that $\dim \GSpin_n = \dim \Sim_n$ and hence $\GSpin_n \cong \Sim_n$ as follows:
\begin{itemize}
\item [\bf n=1:]\label{n=1}  If $n=1$ then by Theorem \ref{CSA}, $C^+(V_1)$ is a central simple algebra over $F$ of dimension 1.  Hence, $C^+(V_1) = F$, and
\begin{align*}
\Sim_1& = \{g \in C^+(V_1):N(g) \in F^{\times}\}\\
&=\{ g \in F: N(g) \in F^{\times}\}\\
&=\{ g \in F: gg^* \in F^{\times}\}\\
&=\{ g \in F: g^2 \in F^{\times}\}\\
&\cong F^{\times}\\
& \cong \GSpin_1.
\end{align*}

\item [\bf n=2:] \label{n=2} If $n=2$ then by Theorem \ref{CSA} there are two cases to consider.

Case 1: Assume $d=1$, so that $V_2=\mathbb{H}$ (hyperbolic plane). Then $C^+(V_2) = A \times A$, where $A$ is a central simple algebra over $F$ with $\dim_F A = 1$, $C^+(V_2) = F \times F$.  Then by Theorem \ref{CSA} we know that the involution $*$ is given by $(a, b)^*=(b, a)$ for $(a, b)\in F\times F$. The fixed field of this involutions is $\Delta F=\{(a,a)\}$, which is equal to $C^0(V_2)$. Hence
\begin{align*}
\Sim_2& = \{(a,b) \in F \times F :(a,b)(a,b)^* \in \Delta F^{\times}\}\\
&=\{(a,b) \in F \times F:(a,b)(b,a) \in \Delta F^{\times}\}\\
&=\{(a,b) \in F \times F:(ab,ab) \in \Delta F^{\times}\}\\
&\cong F^{\times} \times F^{\times} \cong {\GL}_1 \times {\GL}_1\\
&\cong\GSpin_2.
\end{align*}

Case 2: Assume $d\neq 1$, so that $V_2=E$ and $q_2=N_{E/F}$. Then $C^+(V_2)$ is a central simple algebra over $E=$ with $\dim_F C^+(V_2)=2$, which implies  $C^+(V_2)=E$. Again by Theorem \ref{CSA} the involution $*$ is the Galois conjugate of $g$. Then
\begin{align*}
\Sim_2& = \{g \in E:g g^* \in F^{\times}\}\\
&\cong E^{\times} \cong \Res_{E/F} {\GL}_1\\
&\cong \GSpin_2.
\end{align*}

\item [\bf n=3:] \label{n=3}
If $n=3$ then $C^+(V_3)$ is a central simple algebra over $F$ and $\dim_FC^+(V_3) = 4$, which means $C^+(V_3)$ is a quaternion algebra $D$ over $F$. Moreover, by Theorem \ref{CSA}, the involution $*$ is symplectic, which implies that $*$ is the quaternion conjugation.  Then
\begin{align*}
\Sim_3& = \{ g \in C^+(V_3) :N(g)\in F^{\times}\}\\
& =\{ g \in D:g\overline{g} \in F^{\times}\}\\
&=D^{\times}\\
& \cong \GSpin_3,
\end{align*}
where the bar indicates the quaternion conjugation. In particular, if $D$ is split then $\GSpin_3=\GL_2$.

\item [\bf n=4:]\label{n=4} If $n=4$ then by Theorem \ref{CSA} there are two cases to consider.

Case 1: Assume $d=1$, so that $C^+(V_4) = A \times A$ where $A$ is a central simple algebra over $F$ with $\dim_F A = 4$. Then $C^+(V_4)=D\times D$ where $D$ is a quaternion algebra over $F$, and the involution $*$ is symplectic on each factor $D$. Hence the involution $*$ is given by
\[
(x, y)^*=(\overline{x}, \overline{y})
\]
for $(x, y)\in D$, where the bar is the quaternion conjugation as above. Hence
\begin{align*}
\Sim_4 &= \{(x,y) \in D \times D :(x,y)(\overline{x},\overline{y}) \in \Delta F^{\times}\}\\
&=\{(x,y) \in D \times D :(x\overline{x},y\overline{y}) \in \Delta F^{\times}\}\\
&=\{(x,y) \in D \times D :x\overline{x}=y\overline{y} \in F^{\times}\}\\
&=\{(x,y) \in D \times D :N_D(x) = N_D(y) \in F^{\times}\}\\
&\cong  \GSpin_4,
\end{align*}
where $N_D$ is the reduced norm on $D$.  In particular, if $D=M_2(F)$ then $$\GSpin_4 = \{(x,y) \in {\GL}_2(F) \times {\GL}_2(F): \det x = \det y \}.$$

Case 2:  Assume $d\neq 1$, so that $C^+(V_4)$ is central simple algebra over $E=F(\sqrt{d})$ and $\dim_FC^+(V_4) = 8$ so $\dim_E C^+(V_4) = 4$.  Thus, $C^+(V_4)$ is a quaternion algebra $D_E$ over $E$  and the involution * is the quaternion conjugation on $D_E$, which fixes $E$ point-wise.
Then
\begin{align*}
 \Sim_4 &= \{g\in C^+(V_5):gg^* \in F^{\times}\}\\
 &=\{g \in D_E:N_{D_E}(g) \in F^{\times}\}\\
 &\cong \GSpin_4,
\end{align*}
where $N_{D_E}$ is the reduced norm on the quaternion algebra $D_E$ and the isomorphism follows from the equal dimensions of the respective Lie algebras of $\{g \in D_E:N_{D_E}(g) \in F^{\times}\}$ and $\GSpin_4$. In particular, if $D_E = M_2(E)$, then $N_{D_E}(g)=\det g$, so $\GSpin_4 = \{g \in {\GL}_2(E) : \det g \in F^{\times}\}$.

\item [\bf n=5: ]\label{n=5}  If $n=5$ then  $C^+(V_5)$ is a central simple algebra of dimension 16 over $F$. Now, for our purposes we need only the case when $V_5$ is of the form  $$V_5 = \mathbb{H} \perp \mathbb{H} \perp \langle a \rangle$$ where $a \in F^{\times}$. Then we will show $C^+(V_5)=M_4(F)$ and $\GSpin_5=\GSp_4$. First assume $a=-1$, so that $V_5 = \mathbb{H} \perp \mathbb{H} \perp \langle -1 \rangle$. Then by \cite[Cor 2.10, pg. 112]{Lam05},
\begin{align*}
C^+(V_5)&=C^+(\langle -1 \rangle \perp (\mathbb{H} \perp \mathbb{H}))\\
&= C(\mathbb{H} \perp \mathbb{H})\\
&=M_4(F).
\end{align*}
For general $a \in F^{\times}$, we have $aV_5 = \mathbb{H} \perp \mathbb{H} \perp \langle -a \rangle$, and $C^+(aV_5)=C^+(V_5)$ by \cite[Cor 2.11, pg. 112]{Lam05}.  Hence, $C^+(V_5)=M_4(F).$

The involution $*$ on $C^+(V_5)$ is symplectic involution, which by definition means that there exists a 4-dimensional symplectic space $(V,\langle-,-\rangle)$ over $F$ with $\dim V=4$ and an isomorphism $M_4(F)\cong\End_FV$ such that the involution $*$ is the pullback of the adjoint involution on $V$ induced by the symplectic form $\langle-,-\rangle$. Hence we have
\begin{align*}
\Sim_5& = \{ g \in M_4(F):gg^* \in F^{\times}\}\\
&= \{ g\in\End_F V = gg^* \in F^{\times}\},
\end{align*}
where by $g^*$ for $g\in\End_FV$ we mean the adjoint with respect to the symplectic form on $V$. Then viewed inside $\End_FV$ we have
\begin{align*}
\langle gv, gv'\rangle&=\langle v, g^*gv'\rangle=N(g)\langle v, v'\rangle
\end{align*}
for all $v, v'\in V$, which implies $\Sim_5 \cong {\GSp}_4,$ and $N(g)$ is the similitude factor of $g\in\GSp_4$. The equality of dimensions of $\Sim_4 \cong \GSp_4$ and $\GSpin_5$ gives  $$\GSpin_5 \cong \GSp_4.$$
  \end{itemize}

\begin{remark}
Sometimes in the literature, $\GSpin_1$ is ''defined" as $\GL_1$. (See, for example, \cite[pg. 678]{Asg02}.) However, this actually ''follows from" our definition of $\GSpin$. Also the case $n=4$ is also proven in \cite[Proposition 2.1]{AC16} by using roots and coroots.)
\end{remark}

\subsection{On certain $L$-functions}
In this subsection, we review the basics of the $L$-group of $\GSpin$ and certain $L$-functions attached to a cuspidal automorphic representation of $\GSpin$. Accordingly, we let $F$ be a number field and we simply write $\GSpin_n=\GSpin(V_n)$.

First recall that the Langlands dual group $\widehat{\GSpin}_n$ is defined as
\[
\widehat{\GSpin}_n=
\begin{cases}
\GSp_{2m}(\C)&\text{if $n=2m+1$};\\
\GSO_{2m}(\C)&\text{if $n=2m$}.
\end{cases}
\]
Next assume that $\GSpin_n$ is quasi-split. Then (the Galois form of) the global $L$-group $\leftidx{^}L\GSpin_n$ is defined as
\[
\leftidx{^}L\GSpin_n=
\begin{cases}
\GSp_{2m}(\C)\times \Gamma_F&\text{if $n=2m+1$};\\
\GSO_{2m}(\C)\rtimes \Gamma_F&\text{if $n=2m$},
\end{cases}
\]
where $\Gamma_F$ is the absolute Galois group of $F$. Note that for $n=2m$ the action of $\Gamma_F$ is trivial on $\Gamma_E$ where $F=E(\sqrt{d})$ and $d=\disc(V_n)$, so that we have the natural surjection
\[
\GSO_{2m}(\C)\rtimes \Gamma_F\longrightarrow \GSO_{2m}\rtimes\Gal(E/F)\subseteq\GO_{2m}(\C),
\]
where in case $d\neq 1$ the nontrivial element in $\Gal(E/F)$ acts as in, say \cite[Section 4.3]{HS16}, so that we have the inclusion in $\GO_{2m}(\C)$. If $\GSpin_n$ is not quasi-split, there exists a unique quasi-split inner form $\GSpin_n^*$ of $\GSpin_n$, and we define
\[
\leftidx{^}L\GSpin_n=\leftidx{^}L\GSpin_n^*.
\]
Now for each place $v$ of $F$, we define the local $L$-group $\leftidx{^}L\GSpin_n(F_v)$ analogously as above by replacing $F$ by $F_v$.

Let $\pi=\otimes_v\pi_v$ be a cuspidal automorphic representation of $\GSpin_n(\A)$. Assuming the (conjectural) local Langlands correspondence for $\GSpin_n$, we have the local Langlands parameter
\[
\phi_{\pi_v}:WD_{F_v}\longrightarrow\leftidx{^}L\GSpin_n(F_v),
\]
where $WD_{F_v}$ is the Weil-Deligne group of $F_v$. Now for each homomorphism
\[
\rho:\leftidx{^}L\GSpin_n(F_v)\longrightarrow\GL_N(\C),
\]
the local $L$-factor is defined as
\[
L_v(s, \pi_v, \rho):=L_v(s, \rho\circ\pi_{\phi_v}),
\]
where the right-hand side is the local $L$-factor of Artin type associated with the $N$-dimensional Galois representation $\rho\circ\pi_{\phi_v}$. We then define the global automorphic $L$-function by
\[
L(s, \pi, \rho)=\prod_vL_v(s, \rho\circ\phi_{\pi_v}),
\]
and of course it is expected that this product converges for $\Re(s)$ sufficiently large and admits meromorphic continuation and a functional equation.

There are three cases of $\rho$ we are interested in: the standard representation, adjoint representation and tensor product representation, where the last one actually involves another $\GSpin_{n+1}$.

Firstly, we have the standard representation $\rho=std_n$, which is given by
\[
std_n:\leftidx{^}L\GSpin_n(F_v)\longrightarrow\GL_{2m}(\C),
\]
where $m$ is such that $n=2m+1$ or $n=2m$, depending on the parity of $n$. To be more precise, we have the natural maps
\[
std_n:\leftidx{^}L\GSpin_{2m+1}(F_v)=\GSp_{2m}(\C)\times \Gamma_{F_v}\to\GSp_{2m}(\C)\subseteq\GL_{2m}(\C),
\]
and
\[
std_n:\leftidx{^}L\GSpin_{2m}(F_v)\to\GSO_{2m}\rtimes\Gal(E_v\slash F_v)\subseteq\GO_{2m}(\C)\subseteq\GL_{2m}(\C),
\]
where $E_v$ is interpreted as $F_v\times F_v$ for split $v$. We have the $L$-functions $L(s, \pi, std_n)=\prod_vL(s, \pi_v, std_n)$ called the standard $L$-function, which we simply write $L(s, \pi)$.

Secondly, we need to consider the adjoint representation. Note that we have the adjoint representations
\[
\Ad_{\mathfrak{gsp}}:\leftidx{^}L\GSpin_{2m+1}(F_v)\to \GSp_{2m}(\mathbb{C})\to\GL(\mathfrak{gsp}_{2m}(\C))
\]
and
\[
\Ad_{\mathfrak{gso}}:\leftidx{^}L\GSpin_{2m}(F_v)\to\GO_{2m}(\C)\to \GL(\mathfrak{gso}_{2m}(\C)),
\]
where $\mathfrak{gsp}_{2m}(\C)$ and $\mathfrak{gso}_{2m}(\C)$ are the Lie algebras of the corresponding groups as usual. Since
\[
\mathfrak{gsp}_{2m}(\mathbb{C})=\mathfrak{sp}_{2n}(\mathbb{C}) \oplus \mathbb{C}\qand
\mathfrak{gso}_{2m}(\mathbb{C})=\mathfrak{so}_{2n}(\mathbb{C}) \oplus \mathbb{C},
\]
we have
\[
\Ad_{\mathfrak{gsp}}=\Ad_{\mathfrak{sp}} \oplus \mathds{1}\qand
\Ad_{\mathfrak{gso}}=\Ad_{\mathfrak{so}} \oplus \mathds{1}.
\]
Accordingly, we have
\[
L_v(s,\pi_v,\Ad_{\mathfrak{gsp}})=L_v(s,\pi_v,\Ad_{\mathfrak{sp}})\zeta_v(s)\qand
L_v(s,\pi_v,\Ad_{\mathfrak{gso}})=L_v(s,\pi_v,\Ad_{\mathfrak{so}})\zeta_v(s),
\]
which, of course, gives
\[
L(s,\pi,\Ad_{\mathfrak{gsp}})=L(s,\pi,\Ad_{\mathfrak{sp}})\zeta(s)\qand
L(s,\pi,\Ad_{\mathfrak{gso}})=L(s,\pi,\Ad_{\mathfrak{so}})\zeta(s),
\]
by taking the product over all $v$. Then we define
\[
L(s,\pi,\Ad):=\begin{cases}L(s,\pi,\Ad_{\mathfrak{sp}})&\text{if $n=2m+1$};\\
L(s,\pi,\Ad_{\mathfrak{so}})&\text{if $n=2m$},
\end{cases}
\]
and call it the adjoint $L$-function of $\pi$.  The adjoint $L$-function $L(s,\pi,\Ad)$ is conjecturally non-zero and holomorphic as in \cite[(3.2) and pg. 483]{LM15}.

Thirdly, we consider the tensor product representation. For this we also consider $\GSpin_{n+1}$ and the standard representation $std_{n+1}$. Then we have the tensor product $std_n\otimes std_{n+1}$ of the two standard representations, which is of the form
\[
std_n\otimes std_{n+1}:\leftidx{^}L\GSpin_n(F_v)\times \leftidx{^}L\GSpin_{n+1}(F_v)\longrightarrow
\begin{cases}
\GL_{n^2}(\C), &\text{if $n=2m$};\\
\GL_{(n-1)n}(\C),&\text{if $n=2m+1$}.
\end{cases}
\]
Then for cuspidal automorphic representations $\pi_n$ and $\pi_{n+1}$ of $\GSpin_n(\A)$ and $\GSpin_{n+1}$, respectively, we write
\[
L(s, \pi_n\times\pi_{n+1})=L(s, \pi_n\boxtimes\pi_{n+1}, std_n\otimes std_{n+1})
\]
and call it the tensor product $L$-function.
Let us mention that the $L$-function $L(s, \pi_n\times\pi_{n+1})$ is conjecturally holomorphic, see for example \cite[Theorem 5.1] {PSR87}.
\subsection{Component groups for Langlands parameters}
The goal of this section is to discuss the component groups for the (conjectural) global Langlands parameters in relation to Conjecture \ref{C:conj2}. As we did in the introduction, let $\mathcal{L}_F$ be the hypothetical global Langlands group of our number field $F$ and let $\phi: \mathcal{L}_F \rightarrow \leftidx{^}LG = \widehat{G} \rtimes \Gamma_F$ be a global Langlands parameter. Set $S_{\phi}=\Cent(\Im(\phi),\widehat{G})$, and define
$$\mathcal{S}_{\phi}=S_{\phi}/S_{\phi}^0Z(\widehat{G})^{\Gamma_F},$$
where $S^0_{\phi}$ is the identity component of the complex reductive group $S_{\phi}$, $Z(\widehat{G})$ is the center of $\widehat {G}$ and $Z(\widehat{G})^{\Gamma_F}$ is the subgroup of invariants in $Z(\widehat{G})$ under the natural action of $\Gamma_F$.

Now, let $\pi_n$ and $\pi_{n+1}$ be cuspidal automorphic representations of $\GSpin_n(\A)$ and $\GSpin_{n+1}(\A)$, respectively. We then have made the following conjecture  (Conjecture \ref{C:conj2})
$$2^{\beta}=4|\mathcal{S}_{\phi_n}||\mathcal{S}_{\phi_{n+1}}|,$$
where $\phi_n$ and $\phi_{n+1}$ are the (conjectural) global Langlands parameters of $\pi_n$ and $\pi_{n+1}$, respectively.

\begin{lemma}\label{component_group_abelian_2_group}
The group $\mathcal{S}_{\phi_n}$ is an elementary abelian $2$-group, so in particular its order is a power of $2$.
\end{lemma}
\begin{proof}
First consider the case $n=2m+1$, so that the (conjectural) global Langlands parameter $\phi_n$ is of the form
$$\phi_n: \mathcal{L}_F \longrightarrow \leftidx{^}L{\GSpin}_{2m+1}={\GSp}_{2m}(\mathbb{C}) \times \Gamma_F$$ such that $\Im(\phi_{n})$ is not in a proper parabolic subgroup of $\GSp_{2m}(\mathbb{C})$. Since the action of $\Gamma_F$ on ${\GSp}_{2m}(\mathbb{C})$ is trivial, we may consider $\phi_n$ as the composite
\[
\phi_n: \mathcal{L}_F \longrightarrow {\GSp}_{2m}(\mathbb{C}) \times \Gamma_F\longrightarrow \GSp_{2m}(\C),
\]
where the second map is the obvious projection. Furthermore, by writing $\GSp_{2m}(\C)=\GSp(W)$, where $(W,\langle-,-\rangle)$ is an $m$-dimensional symplectic space over $F$, we view $\phi_n$ as a representation of $\mathcal{L}_F$ acting on $W$.

Now, let $W_1 \subseteq W$ be an irreducible subspace of $\phi_n$, and $W_1^{\perp}$ the orthogonal complement of $W_1$ with respect to the symplectic form $\langle-,-\rangle$. One can readily see that $W_1^{\perp}$ is a subrepresentation of $\phi_n$.

Assume $W_1 \cap W_1^{\perp}\neq 0$. Then $W_1 \cap W_1^{\perp}$ is a nonzero subrepresentation of $W_1$, which, by irreducibility of $W_1$, implies that $W_1 \cap W_1^{\perp}=W_1\subseteq W_1^\perp$. Hence $W_1$ is totally isotropic, and so the group $\Im(\phi_m)$ stabilizes the flag $0 \subseteq W_1 \subseteq W$ which implies $\Im(\phi_n)$ is in the proper parabolic subgroup $P(W_1) \subseteq \GSp(W)$, which contradicts to our assumption that $\Im(\phi_n)$ is not in a proper parabolic subgroup of $\GSp(W)$.

Thus we necessarily have $W_1 \cap W_1^{\perp}=0$, so that $W = W_1 \oplus W_1^{\perp}$. By induction
$$W= W_1 \oplus W_2 \oplus \cdots \oplus W_k$$
where each $(W_i,\langle \cdot, \cdot, \rangle)$ is a smaller symplectic space.  Then $$\Im(\phi_n) \subseteq (\GSp(W_1) \times \cdots \times \GSp(W_k)) \cap \GSp(W),$$
which forces $$\Im(\phi_n) \subseteq \{(g_1,\cdots,g_k) \in \GSp(W_1) \times \cdots \times \GSp(W_k):N_1(g_1)= \dots = N_k(g_k) \},$$
where each $N_i$ is the similitude character on $\GSp(W_i)$.
Then we see that
\begin{align}\label{S_phi_n}
S_{\phi_n}&=\Cent(\Im(\phi_{\pi}),{\GSp}_{2m}(\mathbb{C}))\\\notag
&=\{(a_1I_{W_1},\dots, a_kI_{W_k}):a_1^2=\dots =a_k^2 \}\\\notag
&=\{(\pm aI_{W_1},\cdots, \pm aI_{W_k}): a\in \mathbb{C}^{\times}\}\notag
\end{align}
where $I_{W_i}$ is the identity on $W_i$. The identity component is then
$$S_{\phi_n}^{\circ}=\{( aI_{W_1},\cdots,  aI_{W_k}): a\in \mathbb{C}^{\times}\}=Z(\GSp_{2m}(\mathbb{C})).$$
Hence we have $$S_{\phi_n}/S_{\phi_n}^{\circ}Z(\hat{G})^{\Gamma}=\{ (\pm 1, \dots,\pm 1) \}.$$
Finally, we have that
\begin{align*}
\mathcal{S}_{\phi_n}&=\{ (\pm 1, \dots,\pm 1) \}\\
&= (\mathbb{Z}/2\mathbb{Z})^{k},
\end{align*}
which finishes the proof for $n=2m+1$.

Next assume $n=2m$. Then the Langlands parameter is of the form
\[
\phi_n: \mathcal{L}_F \longrightarrow \leftidx{^}L{\GSpin}_{2m+1}={\GSO}_{2m}(\mathbb{C}) \rtimes \Gamma_F.
\]
Note that $\Gamma_F$ acts trivially on the center $Z(\GSO_{2m})$. Furthermore, as we have seen in the previous subsection, the action of $\Gamma_F$ is trivial on $\Gamma_E$, where $E=F(\sqrt{d})$, and hence we may consider $\phi_n$ as a map
\[
\phi_n: \mathcal{L}_F \longrightarrow{\GSO}_{2m}(\mathbb{C}) \rtimes \Gal(E/F)=\GO_{2m}(\C).
\]
Again by writing $\GO_{2m}(\C)=\GO(W)$ for a complex symmetric bilinear space $W$, we can consider $\phi_n$ as a representation of $\mathcal{L}_F$ acting on $W$. Then we can argue as before.
\end{proof}

Next, let us set up some general notation. Let $G$ be a reductive group over $F$. Set $\widehat{G}_{der}$ to be the derived group of $\widehat{G}$ and $\widehat{G}_{sc}$ the simply connected cover of $\widehat{G}_{der}$, so that we have the maps
\begin{equation}\label{simply_connected_G}
\widehat{G}_{sc}\twoheadrightarrow\widehat{G}_{der}\subseteq \widehat{G} \rightarrow \widehat{G}/Z(\widehat{G})^{\Gamma}.
\end{equation}
For each global Langlands parameter $\phi:\mathcal{L}_F\to\widehat{G}\rtimes\Gamma_F$, we set $S_{\phi, sc}\subseteq \widehat{G}_{sc}$ to be the full preimage of $S_{\phi}\subseteq \widehat{G}$ under the map $\widehat{G}_{sc}\to \widehat{G}$ as above. We then define the larger component group by
$$\mathcal{S}_{\phi,sc}:=S_{\phi,sc}/S^0_{\phi,sc},$$ which is a central extension of
$\mathcal{S}_{\phi}$ by
\begin{equation}\label{E:Z_hat}
\widehat{Z}_{\phi,sc}=\widehat{Z}_{sc}/\widehat{Z}_{sc} \cap S_{\phi,sc}^{\circ},
\end{equation}
where $\widehat{Z}_{sc}$ is the center of $\widehat{G}_{sc}$ and $S_{\phi,sc}$ is the full inverse image of $S_{\phi}$ in $\widehat{G}_{sc}$; namely we have the short exact sequence
\begin{equation}\label{centralext2}
\begin{tikzcd}
1 \arrow[r] & \widehat{Z}_{\phi,sc} \arrow[r] & \mathcal{S}_{\phi,sc} \arrow[r]  & \mathcal{S}_{\phi} \arrow[r] & 1.
\end{tikzcd}
\end{equation}
(See, for example, \cite[(9,2.2)]{Art13}.) It should be noted that this immediately implies
\begin{equation}\label{order_of_component_group}
|\mathcal{S}_{\phi,sc}|=|\widehat{Z}_{\phi,sc}||\mathcal{S}_{\phi}|.
\end{equation}

Let us note that if $G=\GSpin_{2m+1}$ then $$\widehat{G}={\GSp}_{2m}(\mathbb{C}), \hspace{.25 in}\widehat{G}_{der}={\Sp}_{2m}(\mathbb{C}), \hspace{.25 in} \widehat{G}_{sc} ={\Sp}_{2m}(\mathbb{C}),$$
and if $G=\GSpin_{2m}$ then $$\widehat{G}={\GSO}_{2m}(\mathbb{C}),\hspace {.25 in} \widehat{G}_{der}={\SO}_{2m}(\mathbb{C}), \hspace {.25 in} \widehat{G}_{sc} ={\Spin}_{2m}(\mathbb{C}).$$

With this said, we first have the following.
\begin{lemma}
Assume $G=\GSpin_n$. In the above notation, we have
\[
|\mathcal{S}_{\phi, sc}|=\begin{cases} 2|\mathcal{S}_{\phi}|&   \text{ if } n=2m+1 \\
4|\mathcal{S}_{\phi}|& \text{ if } n=2m\\
\end{cases}
\]
where $\phi$ is, of course, a hypothetical global Langlands parameter for $\GSpin_n$.
\end{lemma}
\begin{proof}
Assume $n=2m+1$. Then \eqref{simply_connected_G} is written as
\[
\Sp_{2m}(\C)\twoheadrightarrow\Sp_{2m}(\C)\subseteq\GSp_{2m}(\C).
\]
From \eqref{S_phi_n}, one can see that $S_{\phi,sc}=S_{\phi}\cap\Sp_{2m}(\C)=\{ \pm 1, \pm1,\dots,\pm 1\}$. So $S_{\phi,sc}^{\circ}=1$.  Hence, we can see
\[
\widehat{Z}_{sc}=Z(\Sp_{2m}(\C))=\mu_2\qand \widehat{Z}_{sc} \cap S_{\phi,sc}^{\circ}=1.
\]
Thus \eqref{E:Z_hat} and \eqref{order_of_component_group} prove the lemma.

Assume $n=2m$. Then \eqref{simply_connected_G} is written as
\[
\Spin_{2m}(\C)\twoheadrightarrow\SO_{2m}(\C)\subseteq\GSO_{2m}(\C).
\]
Hence from \eqref{S_phi_n} one can see
\[
S_{\phi}\cap\SO_{2m}(\C)=\{ \pm 1, \pm 1,\dots,\pm 1\},
\]
which gives
\[
S_{\phi,sc}^{\circ}=1.
\]
Thus, since $Z(\Spin_{2m}(\C))=\mu_2\times\mu_2$ or $\mu_4$, we have the lemma from \eqref{E:Z_hat} and \eqref{order_of_component_group}.
\end{proof}

This lemma immediately implies the following, which is a part of Conjecture \ref{C:conj2} in the introduction.
\begin{corollary}
For the (conjectural) global Langlands parameters $\phi_n$ and $\phi_{n+1}$ of $\pi_n$ and $\pi_{n+1}$, respectively, we have
$$4|\mathcal{S}_{\phi_n}||\mathcal{S}_{\phi_{n+1}}|=\dfrac{1}{2}|\mathcal{S}_{\phi_n,sc}||\mathcal{S}_{\phi_{n+1},sc}|.$$
\end{corollary}

\section{Local Integrals of Matrix Coefficients}
In this section, we take up the local intertwining map $\alpha_{\omega_v}^\natural$ in \eqref{E:local_integral_alpha} and firstly prove that the integral that defines $\alpha_{\omega_v}^\natural$ converges at every $v$, assuming the representations are tempered, and then secondly compute the integral for unramified data, which essentially follows from \cite{II10}.

In this section, everything is purely local, and hence we suppress the subscript $_v$ from our notation, and in particular $F$ will denote a local field.

\subsection{Some general lemmas}
In this first subsection, we will prove a couple of lemmas in harmonic analysis, which apply to any connected reductive group over local $F$. (Though those two lemma might be known to experts, we will give our proofs here because we are not able to locate them in the literature.) Accordingly, in this subsection we let $G$ be any connected reductive group over $F$.

First, as usual, we define
\[
G^{1} := \bigcap_{\chi \in \Rat(G)} \ker|\chi|_F
\]
where $\Rat(G)$ is the set of all rational characters on $G$. Then we have the following.
\begin{lemma}\label{conv}
Let $G$ be a reductive group and let $Z^{\circ}$ be the identity component of the center of $G$. Let $f: G \rightarrow \mathbb{C}$ be a measurable function such that $f(zg)=f(g)$ for all $z\in Z^\circ$ and $g\in G$. Then
$\displaystyle\int_{Z^{\circ}\backslash G} f(g)dg$ converges absolutely if and only if $\displaystyle\int_{G^1} f(g)dg$ converges absolutely.
\end{lemma}
\begin{proof}
Set $Z^1=Z^{\circ} \cap G^1$ and let $pr:Z^{\circ}\backslash G\to G^1Z^{\circ} \backslash G$ be the natural projection. Then one can readily see that $\ker(pr)=Z^\circ\backslash G^1Z^\circ\cong Z^1\backslash G^1$, which implies
$$G^1{Z^{\circ}} \backslash G \cong \left(Z^1\backslash G^1 )\right\backslash (Z^{\circ}\backslash G).$$
Thus, we can compute
\begin{align*}\displaystyle\int_{Z^{\circ} \backslash G}f(g)dg&=\displaystyle\int_{(Z^1 \backslash G^1) \backslash (Z^{\circ} \backslash G)} \displaystyle\int_{Z^1 \backslash G^1} f(g_1\overline{g})dg_1d\overline{g}\\
&=\displaystyle\int_{G^1Z^{\circ} \backslash G} \displaystyle\int_{Z^1 \backslash G^1} f(g_1\overline{g})dg_1d\overline{g}\\
&=\displaystyle\sum_{\overline{g} \in G^1Z^{\circ}\backslash G}\displaystyle\int_{Z^1 \backslash G^1} f(g_1\overline{g})dg_1,
\end{align*}
where the last equality follows because $G^1Z^{\circ} \backslash G$ is finite. Hence we have
\[
\int_{Z^{\circ} \backslash G} |f(g)|dg<\infty\quad\text{if and only if}\quad\int_{Z^1\backslash G^1} |f(g)|dg<\infty.
\]

On the other hand, using the invariance of the measure we also have
\begin{align*}
\int_{G^1}f(g)dg & = \displaystyle\int_{Z^1 \backslash G^1}\displaystyle\int_{Z^1}f(ag)dadg\\
&=\displaystyle\int_{Z^1 \backslash G^1}\displaystyle\int_{Z^1}f(g)dadg\\
&=Vol (Z^1)\displaystyle\int_{Z^1 \backslash G^1} f(g)dg,
\end{align*}
since $Z^1 = {Z^{\circ}} \cap G^1$ is compact. Hence
\[
\int_{Z^1\backslash G^1} |f(g)|dg<\infty\quad\text{if and only if}\quad\int_{G^1} |f(g)|dg<\infty.
\]
This finishes the proof.
\end{proof}

To state the second lemma, let us fix a special maximal compact subgroup $K$ of $G$, and a Levi part $M_{\min}$ of a minimal parabolic of $G$. Then we have a Cartan decomposition ${G}=KM^{+}_{\min}K$ where
$$M^{+}_{\min}:=\{m \in M_{\min}:|\alpha_i(m)| \leq 1 \text { for } i=1,\dots,r\}$$
where the $\alpha_i$ correspond to the simple roots. From equation (4) of \cite{Wal03}, we have
\begin{equation}\label{4}
G = {\displaystyle\coprod_{m \in M^+_{\min} / M_{\min}^1}}KmK,
\end{equation}
where
\[
M_{\min}^1:=M_{\min}\cap G^1.
\]
Then we have the following lemma.
\begin{lemma}\label{sil} Let $F$ be non-archimedean.
For $f \in L^1(G)$,
$$\displaystyle\int_{G} f(g)dg = \displaystyle\int_{M_{\min}^{+}}\mu(m)\displaystyle\int_{K \times K} f(k_1mk_2)dk_1dk_2dm$$
where $\mu(m)=\dfrac{Vol(KmK)}{Vol(M_{\min}^1)}=C \cdot Vol(KmK)$, for some positive constant $C$.
\end{lemma}
\begin{proof}
Using  \cite[pg.\ 149]{Sil79} we have
\begin{align*}
\displaystyle\int_{{G}} f(x)dx & =  \displaystyle\sum_{m \in M^{+}_{\min} /M_{\min}^1}\displaystyle\int_{K m K} f(x)dx\\
&=\displaystyle\sum_{m \in M^{+}_{\min} /M_{\min}^1}\displaystyle\int_{K \times K}Vol(KmK)f(k_1mk_2)dk_1dk_2 \\
&=\displaystyle\sum_{m \in M^{+}_{\min} /M_{\min}^1}Vol(KmK)\displaystyle\int_{K \times K} f(k_1mk_2)dk_1dk_2 \\
&=Vol(M_{\min}^1)\displaystyle\sum_{m \in M^{+}_{\min}/M_{\min}^1}\mu(m)\displaystyle\int_{K \times K} f(k_1mk_2)dk_1dk_2\\
&=\displaystyle\int_{M_{\min}^+/M^1_{\min}}\displaystyle\int_{M^1_{\min}}\mu(mm_1)\displaystyle\int_{K \times K} f(k_1mm_1k_2)dk_1dk_2dm_1dm\\
&=\displaystyle\int_{M_{\min}^{+}}\mu(m)\displaystyle\int_{K \times K} f(k_1mk_2)dk_1dk_2dm,
\end{align*}
where we used that $Vol(M^1_{\min})$ is finite because $M^1_{\min}$ is compact. 
 \end{proof}

\begin{remark}

In the archimedean case, a similar integral formula as in Lemma \ref{sil} holds (see, for example, \cite[Th.\ 5.8]{Hel00}).
\end{remark}

\subsection{Convergence of the integral}
By using the two lemmas in the previous subsection we are now in a position to prove the convergence of the integral in (\ref{E:local_integral_alpha}). Hence in this subsection, we specialize to $G=\GSpin(V_n)$, where $(V_n, q_n)$ is an $n$-dimensional quadratic space over $F$. In this case, we have
\[
G^1 = \{ g \in \GSpin(V_n) : |N(g)|=1\}.
\]
Also note that we have a Witt decomposition $V= X \oplus V_n^{an} \oplus Y$, where $X$ and $Y$ are totally isotropic spaces and $V_n^{an}$ is the anisotropic part. By fixing a basis for $X$, we obtain a minimal parabolic subgroup $P_{\min}=M_{\min}N_{\min}$ of $G$ with
\[
M_{\min}= \underbrace{{\GL}_1 \times {\GL}_1 \cdots \times {\GL}_1}_{r -\text{times}}\times \GSpin(V^{an}),
\]
where $r$ is the Witt rank of $V_n$, which is by definition the dimension of $X$. Then one can see that
\[
M^+_{\min}=\{(x_1,x_2,\dots,x_r,g_{an}):|x_i|\leq |x_{i+1}|\}.
\]
We define
\begin{align*}
M^{+,1}_{\min}\:&= \{ m \in M^+_{\min}:|\nu(m)|=1\}\\
&= \{ m \in M^+_{\min}:|\nu(g_{an})|=1\}.\\
\end{align*}
The maximal torus $A_{\min}$ of $M_{\min}$  is of the form
\begin{align*}
A_{\min}&= \underbrace{{\GL}_1 \times {\GL}_1 \cdots \times {\GL}_1}_{r-\text{times}} \times {\GL}_1 \\
&=\{(x_1,x_2,\dots,x_r,x_0)\}.
\end{align*}
We then define
\begin{align*}
A^{+,1}_{\min}:&=\{a \in A_{\min}:|\alpha_i(a)| \leq 1 \text { for } i=1, 2,\dots r \text{ and } a \in G^1\}\\
&=\{(x_1,x_2,\dots, x_r,x_0): |x_i| \leq |x_{i+1}|, 1 \leq i \leq n-1 \text{ and }|x_0|=1\}.
\end{align*}
Let $\delta_{P_{\min}}$ be the modulus character of $P_{\min}$. Then
\[
\delta_{P_{\min}}(x)=\displaystyle\prod_{i=1}^r|x_i|^{d+2r-2i}.
\]

Now, let us get to the integral we would like to show to be convergent. Assume we have  $(V_n, q_n)\subseteq (V_{n+1}, q_{n+1})$, so that we have $\GSpin(V_n)\subseteq\GSpin(V_{n+1})$. For simplicity, we write $G_n=\GSpin(V_n)$ and $G_{n+1}=\GSpin(V_{n+1})$. Let $\pi_i$ be a tempered representation of $G_i$ such that $\omega_{\pi_n}\omega_{\pi_{n+1}}=\omega^2$ for some $\omega$. Then the integrant for $\alpha_{\omega}^{\natural}$ in \eqref{E:local_integral_alpha} is a product of matrix coefficients of $\pi_n$ and $\pi_{n+1}$ together with $\omega^{-1}(N(g))$. Hence the convergence of the integral boils down to the following.
\begin{proposition}\label{P:convergence_local_period}
Keep the above notation and assumption, so in particular assume $\pi_{n}$ and $\pi_{n+1}$ are tempered. Then for all the matrix coefficients $\Phi_{n+1}$ and $\Phi_{n}$ of $\pi_{n+1}$ and $\pi_n$, respectively, the integral
\[
\int_{Z_{n}^{\circ}\backslash G_n }\Phi_{n+1}(g)\Phi_{n}(g) \omega^{-1}(N(g))dg
\]
is absolutely convergent, where recall that $Z_n^\circ=\GL_1$ even when $n=2$.
\end{proposition}

\begin{proof}
Let us first mention that, in this proof, for each $G_i$ we denote the various subgroups introduced above by $P_{i,\min}$, $A_{i,\min}$, $M_{i,\min}^{1}$, $M_{i,\min}^{+,1}$ , $A_{i,\min}^{+,1}$, etc, and also we denote the Witt index of $V_i$ by $r_i$.

Assume $n>2$. Then $Z_n^{\circ}$ is indeed the identity component of the center and so by Lemma \ref{conv} it suffices to show the absolute convergence of
\[
\int_{G_n^1}\Phi_{n+1}(g)\Phi_{n}(g) \omega^{-1}(N(g))dg.
\]
Using \eqref{4} we have
\[
G_n^1 = \left(\coprod_{m \in M^+_{n,\min} / M_{n,\min}^1}K_nmK_n \right)\bigcap G_n^1=\coprod_{m \in (M^+_{n,\min} / M_{n,\min}^1)\cap G_n^1}K_nmK_n.
\]
Then by Lemma \ref{sil}, the convergence of the integral is reduced to the convergence of
\[
\int_{M^{+,1}_{n,\min}}\mu(m)\displaystyle\int_{K_n  \times K_n } |\Phi_{n+1}(k_1mk_2)||\Phi_{n}(k_1mk_2)|dk_1dk_2dm
\]
where $\mu(m)= C \cdot Vol(K_nmK_n)$ for some positive constant $C$.

Furthermore, since $\Phi_{n}$ and $\Phi_{n+1}$ are matrix coefficients of tempered representations, they satisfy for any $g \in G_i$
$$|\Phi_{i}(g)| \leq A\Xi_i(g)(1+\sigma_i(g))^{B}$$
for some positive constants $A$ and $B$, where $\Xi_i$ and $\sigma_i$ are, respectively, Harish-Chandra's spherical function and a height function on $G_i$. (See \cite[pg.\ 274]{Wal03}.) Note that here we may and do assume $\sigma_{n+1}|_{G_n}=\sigma_n$ and simply write $\sigma$ for both. Since both $\Xi_i$ and $\sigma$ are $K_i$-bi-invariant, and $K_n \times K_n$ is compact, the convergence of the integral reduces to the convergence of
$$\displaystyle\int_{M^{+,1}_{n,\min}}\mu(m) \Xi_{n+1}(m)\Xi_{n}(m)(1+\sigma(m))^{2B}dm.$$

By Theorem 4.2.1 in \cite[p.154]{Sil79} and \cite[Lemma II.1.1]{Wal03}, there exist positive constants $A$ and $B$ such that
$$A^{-1}\delta_{P_{i, \min}}^{1/2}(m)\leq \Xi_i(m)\leq  A\delta_{P_{i,\min}}^{1/2}(m)(1+\sigma(m))^{B}$$
for any $m\in M^{+,1}_{i,\min}$. So the convergence of the integral is reduced to the convergence of
 $$\displaystyle\int_{M^{+,1}_{n,\min} }\mu(m)\displaystyle (\delta_{P_{n+1,\min}})^{1/2}(m) (\delta_{P_{n,\min}})^{1/2}(m)(1+\sigma(m))^{2B}dm.$$

Moreover,  there exists a positive constant $A$ such that $\mu(m) \leq A\delta_{P_{n,\min}}^{-1}$ for any $m \in M^{+,1}_{n,\min}$ \cite[page 241]{Wal03}.  So it is enough to show that
$$\displaystyle\int_{M^{+,1}_{n,\min}}(\delta_{P_{n,\min}})^{-1/2}(m)(\delta_{P_{n+1,\min}})^{1/2}(m)(1+\sigma(m))^{2B} dm$$
converges absolutely.

When $n$ is even, $A^{+,1}_{n,\min}$ sits inside of $A_{n+1,\min}^{+,1}$. Thus, the convergence of the integral is reduced to the convergence of
\[
\int_{A_{n,\min}^{+,1}}(\delta_{P_{n,\min}})^{-1/2}(m)(\delta_{P_{n+1,\min}})^{1/2}(m)(1+\sigma(m))^{2B} dm.
\]
Hence, the convergence of the integral is reduced to  the convergence of
$$ \displaystyle\int_{|x_0|=1}\displaystyle\int_{|x_1|\leq \cdots \leq |x_{r_{n}}|\leq 1}|x_1\cdots x_{r_{n}}|^{1/2}\left(1 -\displaystyle\sum_{j=1}^{r_{n}}\log|x_j|\right)^{2B}d^{\times}x_1 \cdots d^{\times}x_{r_{n}}d^{\times}x_0$$
$$= \displaystyle\int_{|x_0|=1}d^{\times}x_0\displaystyle\int_{|x_1|\leq \cdots \leq |x_{r_{n}}|\leq 1}|x_1\cdots x_{r_{n}}|^{1/2}\left(1 -\displaystyle\sum_{j=1}^{r_{n}}\log|x_j|\right)^{2B}d^{\times}x_1 \cdots d^{\times}x_{r_{n}}.$$
Since $\displaystyle\int_{|x_0|=1}d^{\times}x_0$ is an integral over a compact set, the convergence of the integral is reduced to the convergence of $$\displaystyle\int_{|x_1|\leq \cdots \leq |x_{r_{n}}|\leq 1}|x_1\cdots x_{r_{n}}|^{1/2}\left(1 -\displaystyle\sum_{j=1}^{r_{n}}\log|x_j|\right)^{2B}d^{\times}x_1 \cdots d^{\times}x_{r_{n}}$$
which is precisely the integral that Ichino-Ikeda consider in \cite[pg. 1388] {II10}.

When $n$ is odd, $A_{n,\min}^{+,1}$ is not a subset of $A_{n+1,\min}^{+,1}$.  Hence, the convergence of the integral in this case is reduced to the convergence of
 \begin{align*}
 \displaystyle\int_{|x_0|=1}\displaystyle\int_{|x_1| \leq \cdots \leq |x_{r_{n}}|\leq 1}&|x_1\cdots x_{r_{n}}|^{1/2}\left(1 -\displaystyle\sum_{j=1}^{r_{n}}\log|x_j|\right)^{2B}d^{\times}x_1 \cdots d^{\times}x_{r_{n}}\\
& + \displaystyle\int_{|x_1| \leq \cdots \leq |x_{r_{n}-1}| \leq |x_{r_{n}}|^{-1}\leq 1}|x_1\cdots x_{r_{n}-1}x_{r_{n}}^{-1}|^{1/2}\\
&\left(1 -\displaystyle\sum_{j=1}^{r_{n}-1}\log|x_j|+\log|x_{r_{n}}|\right)^{2B}d^{\times}x_1 \cdots d^{\times}x_{r_{n}}d^{\times}x_0
\end{align*}
\begin{align*}
 =\displaystyle\int_{|x_0|=1}d^{\times}x_0\displaystyle\int_{|x_1| \leq \cdots \leq |x_{r_{n}}|\leq 1}&|x_1\cdots x_{r_{n}}|^{1/2}\left(1 -\displaystyle\sum_{j=1}^{r_{n}}\log|x_j|\right)^{2B}d^{\times}x_1 \cdots d^{\times}x_{r_{n}}\\
& + \displaystyle\int_{|x_1| \leq \cdots \leq |x_{r_{n}-1}| \leq |x_{r_{n}}|^{-1}\leq 1}|x_1\cdots x_{r_{n}-1}x_{r_{n}}^{-1}|^{1/2}\\
&\left(1 -\displaystyle\sum_{j=1}^{r_{n}-1}\log|x_j|+\log|x_{r_{n}}|\right)^{2B}d^{\times}x_1 \cdots d^{\times}x_{r_{n}}d^{\times}x_0.
\end{align*}
Since $\displaystyle\int_{|x_0|=1}d^{\times}x_0$ is an integral over a compact set, the convergence of the integral is reduced to the convergence of
 \begin{align*}
 \displaystyle\int_{|x_1| \leq \cdots \leq |x_{r_{n}}|\leq 1}&|x_1\cdots x_{r_{n}}|^{1/2}\left(1 -\displaystyle\sum_{j=1}^{r_{n}}\log|x_j|\right)^{2B}d^{\times}x_1 \cdots d^{\times}x_{r_{n}}\\
& + \displaystyle\int_{|x_1| \leq \cdots \leq |x_{r_{n}-1}| \leq |x_{r_{n}}|^{-1}\leq 1}|x_1\cdots x_{r_{n}-1}x_{r_{n}}^{-1}|^{1/2}\\
&\left(1 -\displaystyle\sum_{j=1}^{r_{n}-1}\log|x_j|+\log|x_{r_{n}}|\right)^{2B}d^{\times}x_1 \cdots d^{\times}x_{r_{n}}.
\end{align*}
This integral is precisely the integral that Ichino-Ikeda consider in \cite[pg. 1388] {II10}.

Lastly, assume $n=2$. In this case, we have seen $G_2=F^\times\times F^\times$ or $E^\times$. If $G_2=E^\times$ then $\GL_1\backslash G_2=E^1$, which is compact, and hence the convergence of the integral is immediate. If $G_2=F^\times\times F^\times$ then $\GL_1\backslash G_2=F^\times$, in which case we can apply the above argument by using the estimate of the matrix coefficient, and indeed the computation is easier and left to the reader.
\end{proof}

\subsection{Calculation of integrals in the unramified case}
In this subsection, we consider the unramified integral. Accordingly, we assume that all the data are unramified. To be precise, we assume
\begin{enumerate}
\item $G_i$ is unramified over $F$;
\item $K_{i}$ is a hyperspecial maximal compact subgroup of $G_i$;
\item $K_n \subset K_{n+1}$;
\item $\pi_{i}$ is an unramified representation of $G_i$;
\item $\displaystyle\int_{K_{i}}dg_{i}=1$.
\end{enumerate}
Furthermore, let $\omega$ be the unramified character such that $\omega^2=\omega_{\pi_{n+1}}\omega_{\pi_n}$. Note that there is a unique such $\omega$.

Then we have the following.
\begin{proposition}\label{P:unramified_case}
Under the above assumptions, let $\phi^\circ\in\pi_{n+1}$ and $f^{\circ}\in\pi_n$ be the spherical vectors such that
\[
\mathcal{B}_{\pi_{n+1}}(\phi^\circ,\phi^{\circ})=1\qand
\mathcal{B}_{\pi_{n}}(f^\circ,f^{\circ})=1.
\]
Then we have
\[
\alpha_{\omega}^{\sharp}(\phi^\circ, f^\circ)=\Delta_{\SO_{n+1}}\cdot\frac{L(1/2,\pi_{n}\times \pi_{n+1} \otimes \omega^{-1})}{L(1,{\pi_{n+1}}, \Ad)L(1,{\pi_{n}}, \Ad)}.
\]
\end{proposition}
\begin{proof}
Let
\[
\Phi^{\circ}_{\pi_{n+1}}(g):=\mathcal{B}_{\pi_{n+1}}(\pi_{n+1}(g)\phi^\circ,\phi^{\circ})\qand
\Phi^{\circ}_{\pi_{n}}(g):=\mathcal{B}_{\pi_{n}}(\pi_{n}(g)f^\circ,f^{\circ});
\]
namely they are the normalized spherical matrix coefficients so that $\Phi^{\circ}_{n+1}(1)=\Phi^{\circ}_{n}(1)=1$. Since $\omega_{\pi_{n+1}}$ and $\omega_{\pi_n}$ are unramified, there exist unique unramified square roots $\omega_{\pi_{n+1}}^{1/2}$ and $\omega_{\pi_n}^{1/2}$. Let us denote
\[
\overline{\pi_{n+1}}:=\pi_{n+1} \otimes \omega _{\pi_{n+1}}^{-1/2} \circ N\qand
\overline{\pi_{n}}:=\pi_{n} \otimes \omega _{\pi_{n}}^{-1/2} \circ N,
\]
which have trivial central characters and hence viewed as representations of $\SO_{n+1}$ and $\SO_{n}$, respectively. Then one can readily see that
\begin{align*}
\Phi^{\circ}_{\pi_{n+1}}(g)\omega_{\pi_{n+1}}^{-1/2}(N(g))
&=\mathcal{B}_{\pi_{n+1}}(\pi_{n+1}(g)\omega_{\pi_{n+1}}^{-1/2}(N(g))\phi^\circ,\phi^{\circ})\\
&=\Phi^{\circ}_{\overline{\pi_{n+1}}}(g),
\end{align*}
where $\Phi^{\circ}_{\overline{\pi_{n+1}}}$ is the normalized spherical matrix coefficient of $\overline{\pi_{n+1}}$ so that $\Phi^{\circ}_{\overline{\pi_{n+1}}}(1)=1$. Similarly, we have
\[
\Phi^{\circ}_{\pi_{n+1}}(g)\omega_{\pi_{n+1}}^{-1/2}(N(g))=\Phi^{\circ}_{\overline{\pi_{n+1}}}(g).
\]
Hence by using Theorem 1.2 in \cite{II10} we have the following:
\begin{align*}
\alpha^{\sharp}_\omega(\phi^\circ, f^\circ)
=&\int_{Z_{n}^\circ\backslash G_n}\Phi^{\circ}_{\pi_{n+1}}(g)\Phi^{\circ}_{\pi_{n+1}}(g)
\omega_{\pi_{n+1}}^{-1/2}(N(g))\omega_{\pi_{n}}^{-1/2}(N(g))\,dg\\
=&\int_{\SO_n}\Phi^{\circ}_{\overline{\pi_{n+1}}}(g)\Phi^{\circ}_{\overline{\pi_{n}}}(g)\,dg\\
=&\Delta_{\SO_{n+1}}\cdot \frac{L(1/2,\overline{\pi_{n}}\times\overline{\pi_{n+1}})}{L(1,\overline{\pi_{n+1}}, \Ad)L(1,\overline{\pi_{n}}, \Ad)}.
\end{align*}

Now one can readily see that
\[
L(s,\overline{\pi_{n}}\times\overline{\pi_{n+1}})
=L(s,\pi_{n}\otimes\omega_{\pi_n}^{-1/2}\times\pi_{n+1}\otimes\omega_{\pi_{n+1}}^{-1/2})
=L(s,\pi_{n}\times\pi_{n+1}\otimes\omega^{-1})
\]
because $\omega_{\pi_n}^{-1/2}\omega_{\pi_{n+1}}^{-1/2}=\omega^{-1}$. Also one can see that
\[
L(s,\overline{\pi_{n+1}}, \Ad)=L(s,\pi_{n+1}, \Ad)\qand L(s,\overline{\pi_{n+1}}, \Ad)=L(s,\pi_{n+1}, \Ad)
\]
by definition of the adjoint $L$-function. The proposition follows.
\end{proof}

\begin{remark}
Although the local calculations in the unramified section follow from \cite{II10}, this is only possible because the square root always exists for the unramified case.  In the ramified or global cases, we may not assume this.
\end{remark}

\section{Wrap-up of the conjecture and low rank cases}
In this section, let us first wrap-up our conjecture and then prove low rank cases. So in this section we let $F$ be a number field and $\A$ the ring of adeles.

\subsection{Wrap-up}
Assume $\pi_n$ and $\pi_{n+1}$ are tempered cuspidal automorphic representations of $G_n(\A)=\GSpin_n(\A)$ and $G_{n+1}(\A)=\GSpin_{n+1}(\A)$ such that there exists a Hecke character $\omega$ with $\omega^2=\omega_{\pi_n}\omega_{\pi_{n+1}}$. Fix the tensor product decompositions $\pi_n=\otimes_v\pi_{n, v}$ and $\pi_{n+1}=\otimes_v\pi_{n+1, v}$ and fix factorizable $\phi=\otimes_v\phi_v$ in $\pi_{n+1}$ and $f=\otimes_vf_v$ in $\pi_n$.

First of all, for almost all $v$, the assumptions of Proposition \ref{P:unramified_case} are satisfied and hence we have
\[
\alpha_{\omega_v}(\phi_v, f_v)=\frac{L(1,{\pi_{n+1}}, \Ad)L(1,{\pi_{n}}, \Ad)}{\Delta_{\SO_{n+1}}\cdot L(1/2,\pi_{n}\times \pi_{n+1} \otimes \omega^{-1})}\alpha_{\omega_v}^{\sharp}(\phi_v, f_v)=1,
\]
where $\phi_v$ and $f_v$ are the spherical vectors as in the previous section. Thus the infinite product
\[
\prod_v\alpha_{\omega_v}(\phi_v, f_v)
\]
is well-defined.

Accordingly, we can and do form the conjecture
\begin{conjecture} [The global GGP conjecture for GSpin.] With the assumptions stated above
and if $\pi_{n+1}$ and $\pi_n$ appear with multiplicity one in the discrete spectrum, then 
\[
\left|\mathcal{P}_\omega(\phi, f)\right|^2 =\frac{\Delta_{\SO_{n+1}}}{2^{\beta}}\frac{L(1/2, \pi_{n}  \times \pi_{n+1} \otimes \omega^{-1})}{L(1, \pi_{n+1} , \Ad)L(1, \pi_{n} , \Ad)}\prod_v \alpha_{\omega_v}(\phi_v, f_v),
\]
where 
\[
\Delta_{\SO_{n+1}}:=\begin{cases} \zeta(2)\zeta(4)\cdots\zeta(2m) & \text{ if } \dim V_{n+1} = 2m+1,\\
\zeta(2)\zeta(4)\cdots \zeta(2m-2) \cdot L(m, \chi_{V_{n+1}}) & \text { if } \dim V_{n+1}=2m,
\end{cases}
\]
where $\chi_{V_{n+1}}$ is the quadratic character associated with $V_{n+1}$.
\end{conjecture}
Here we are assuming the $L$-function $L(s, \pi_n \times \pi_{n+1} \otimes \omega^{-1})$ is holomorphic at $s=1/2$ and the the adjoint $L$-functions $L(s,\pi_{n+1},\Ad)$ and $L(s,\pi_{n+1},\Ad)$ are non-zero and holomorphic at $s=1$.
As discussed in the introduction, we also make the following conjecture regarding the constant $2^{\beta}$: 
\begin{conjecture}\label{conj2}
Let $\phi_{n+1}$ and $\phi_n$ be the (conjectural) global Langlands parameters of $\pi_{n+1}$ and $\pi_n$, respectively. If $\pi_{n+1}$ and $\pi_n$ appear with multiplicity one in the discrete spectrum, then 
$$2^{\beta}=4|\mathcal{S}_{\phi_n}||\mathcal{S}_{\phi_{n+1}}|=\dfrac{1}{2}|\mathcal{S}_{\phi_n,sc}||\mathcal{S}_{\phi_{n+1},sc}|.$$
\end{conjecture}
Now let us note the relation between our conjecture and that of Ichino-Ikeda. Assume $\pi_n$ and $\pi_{n+1}$ both have the trivial central character, so that we can choose $\omega=\mathbf{1}$. Then $\pi_n$ and $\pi_{n+1}$ can be viewed as automorphic representations of $\SO_n(\A)$ and $\SO_{n+1}(\A)$. (Conjecturally, this means that if $\phi_n$ is the global $L$-parameter of $\pi_n$ then the image of $\phi_n$ is already in $\Sp_{2m}(\C)$ or $\operatorname{O}_{2m}(\C)$ depending on the parity of $n$, and similarly for $\pi_{n+1}$.) Then one can see that the tensor product $L$-function $L(s, \pi_n\times\pi_{n+1}\otimes\mathbf{1})$ as the $L$-function for $\GSpin$ is equal to the tensor product $L$-function $L(s, \pi_n\times\pi_{n+1})$ as the $L$-function for $\SO$, and similarly for the adjoint $L$-functions. Hence in this case, our conjecture is precisely that of Ichino-Ikeda. In this sense, our conjecture should be interpreted as a generalization of that of Ichino-Ikeda instead of an analogue of it.

\subsection{Conjecture for $\GSpin_2 \times \GSpin_3$ (Waldspurger Formula Case)}
Let us consider the lowest rank case, so we let $V_2$ and $V_3$ be quadratic spaces of dimensions 2 and 3. Note then that $V_2=E$, where $E$ is a quadratic extension of $F$ equipped with the norm form or $V_2=\mathbb{H}$ (hyperbolic plane), and there exists $V_3=D_0$, where $D_0$ is the set of trace zero elements of a (not necessarily division) quaternion algebra $D$ equipped with the norm form. Recall in Section \ref{iso} we have computed
\begin{align*}
\GSpin_2&=\GSpin(V_2)
=\begin{cases} \Res_{E/F} \GL_1 & \text { if $V_2=E$};\\
\GL_1 \times \GL_1 & \text  { if $V_2=\mathbb{H}$};
\end{cases}\\
{\GSpin}_3&=\GSpin(V_3) = D^{\times}.
\end{align*}

In this subsection, we consider the case $V_2=E$, and assume $D$ is such that we have an inclusion $E\subseteq D_0$, which gives the inclusion $\Res_{E/F} \GL_1\subseteq D^\times$. This is essentially the case treated by Waldspurger in \cite{Wal85} and the resulting formula is normally known as the Waldspurger formula.

So we let $\pi_2$ be a cuspidal automorphic representation on $\GSpin_2(\mathbb{A})=\mathbb{A}^\times_E$, namely a Hecke character $\chi$ on $\mathbb{A}^\times_E$ and let $\pi_3$ be a tempered cuspidal automorphic representation of $D^\times(\mathbb{A})$ such that there exists a Hecke character $\omega$ on $\mathbb{A}$ with
\[
\omega^2=\omega_{\pi_3}\chi|_{{\mathbb{A}^{\times}_F}}.
\]
Consider $\pi_3\otimes\omega^{-1}$, which is an automorphic representation of $D^\times(\A)$ with the central character $\omega_{\pi_3 \otimes \omega^{-1}}=\omega^{-2}\omega_{\pi_3}$, so that
\[
\omega_{\pi_3 \otimes \omega^{-1}}\cdot\chi|_{\mathbb{A}^\times}=1.
\]
Then for each $\phi\in V_{\pi_3}$ and $\chi\in V_{\chi}$, our period integral is
\[
\mathcal{P}_\omega(\phi, \chi)
=\int_{E^\times\backslash\A^\times_E\slash\A^\times}\phi(g)f(g)\omega(N(g))^{-1}\,dg
=\int_{E^\times\backslash\A^\times_E\slash\A^\times}(\phi\cdot\omega^{-1})(g)\chi(g)\,dg,
\]
which is nothing but the period integral considered by Waldspurger for $\pi_3\otimes\omega^{-1}$ and $\chi$. Hence by using the Waldspurger formula, we obtain
\begin{align*}
|\mathcal{P}_{\omega}(\phi, \chi)|^2&=\frac{\zeta(2)L(1/2, BC(\pi_3 \otimes \omega^{-1}) \otimes \chi)}{4L(1, \mu)^2L(1, \pi_3 \otimes \omega^{-1}, \Ad)}\prod_v \alpha_{\omega_v}(\phi_v, \chi_v)\\
&=\frac{\zeta(2)L(1/2, BC(\pi_3 \otimes \omega^{-1}) \otimes \chi)}{4L(1, \chi,\Ad)L(1, \pi_3 , \Ad)}\prod_v\alpha_{\omega_v}(\phi_v, \chi_v) \\
&=\frac{\zeta(2)L(s,\pi_3 \otimes \omega^{-1}\times \chi)}{4L(1, \chi,\Ad)L(1, \pi_3 , \Ad)}\prod_v\alpha_{\omega_v}(\phi_v, \chi_v),
\end{align*}
where for the first equality we used the Waldspurger formula with $\mu$ the quadratic character for the extension $E/F$, and for the last equality we used \cite[pg. 102]{Bump}. This confirms Conjecture \ref{C:conj}. Moreover, $|\mathcal{S}_{\phi_2}|=|\mathcal{S}_{\phi_3}|=1$, so $2^{\beta}=4|\mathcal{S}_{\phi_2}||\mathcal{S}_{\phi_3}|$, confirming Conjecture \ref{C:conj2}.

\begin{remark}
Waldspurger assumed the central character of $\pi$ is trivial, but the authors of \cite{YZZ13} removed this condition. Also the Waldspurger formula we used in the above looks slightly different from the original in \cite{Wal85} or from the Waldspurger formula listed in \cite[Theorem 1.4 in Section 1.4.2]{YZZ13}. This is due to the following: Waldspurger chose the global Haar measure used in the period integral such that the $Vol(E^{\times}\backslash \mathbb{A}^{\times}_E/\mathbb{A}^{\times})=2L(1,\mu)$.  The  authors of \cite{YZZ13} chose the global Haar measure such that the volume is 1.  Then each chose local measures to be compatible with their choice of global measure.  With our choice of measures, our formulation is equivalent.
\end{remark}

\subsection{Conjecture for $\GSpin_2\times\GSpin_3$ (Jacquet-Langlands Case)}\label{JLC}
Next consider the case $V_2=\mathbb{H}$, so that $\GSpin_2=\GL_1\times\GL_1$. In this case we have an embedding $\GSpin_2\subseteq\GSpin_3$ only when $\GSpin_3=\GL_2$; namely $D$ in the previous subsection is split.

Before moving on, let us mention that this case is actually excluded from our conjecture in the first place. Indeed, as we will see, even though we can obtain a similar formula by using the well-known Jacquet-Langlands theory, the resulting formula is not exactly as in our conjecture. We consider this case merely as a low rank exception.

Now let $\pi_2$ be a tempered cuspidal automorphic representation of $\GL_1(\A) \times \GL_1(\A)$, so that we have $\pi_2=\chi_1 \boxtimes\chi_2$ where $\chi_1$, $\chi_2$ are both unitary Hecke characters of $\A^\times$, and let $\pi_3$ be a tempered cuspidal automorphic representation of $\GL_2(\A)$ with central character $\omega_{\pi_3}$.  By our assumption, there exists a Hecke character $\omega$ such that
\[
\omega^2=\chi_1\chi_2\omega_{\pi_3}.
\]
\begin{proposition}\label{P:Jacquet-Langlands case}
Keep the above notation and assumption.  Then for $\phi\in V_{\pi_3}$ and $\chi_1 \boxtimes \chi_2 \in V_{\pi_2}$, we have
\[
|\mathcal{P}_\omega(\phi,\chi_1 \boxtimes \chi_2)|^2=\frac{\zeta(2)L(1/2,\pi_3  \times \pi_2  \otimes \omega^{-1})}{2|\mathcal{S}_{\phi_2}||\mathcal{S}_{\phi_3}|L(1,\pi_3,\Ad)}\displaystyle\prod_{v }\alpha_{\omega_v}(\phi_v, \chi_{1,v} \boxtimes \chi_{2,v}),
\]
\end{proposition}
\begin{proof} This is a standard exercise using the well-known Jacquet-Langlands theory as well as  \cite[Proposition 6, pg. 208]{Wal85}
and that $|\mathcal{S}_{\phi_2}|=|\mathcal{S}_{\phi_3}|=1$.  The details are are left to the reader. 
\end{proof}
Note that Proposition \ref{P:Jacquet-Langlands case} is similar to Conjectures \ref{C:conj} and \ref{C:conj2}.

\subsection{Conjecture for $\GSpin_4 \times \GSpin_3$ (Triple Product Formula)}
In this section we prove the conjecture for $\GSpin_4$ and $\GSpin_3$, which essentially boils down to Ichino's triple product formula. So we let $V_3$ and $V_4$ be quadratic spaces of dimension three and four, respectively and write $G_3=\GSpin(V_3)$ and $G_4=\GSpin(V_4)$. Then there exists a (not necessarily division) quaternion algebra $D$ such that 
\begin{align*}
G_3(\A)&=D^\times(\A);\\
G_4(\A)&=\begin{cases}\{(g_1,g_2) \in D^{\times}(\A) \times D^{\times}(\A): N_D(g_1)=N_D(g_2) \in \A^{\times}\};\\
\{g\in D^\times(\A_E) : N(g)\in \A^\times\},
\end{cases}
\end{align*}
where for $G_4$ the first is the case if $\disc(V_4)=1$ and the second is the case if $\disc(V_4)\neq 1$. Here, to be more precise, we are assuming that $V_3$ and $V_4$ are such that the corresponding $G_3$ and $G_4$ are as above with the same $D$, so that we have the inclusion $G_3(\A)\subseteq G_4(\A)$. 

To utilize Ichino's triple product formula, we need to introduce the group
\[
\tilde{G}_4(\A)= \begin{cases} D^{\times}(\mathbb{A}) \times D^{\times}(\mathbb{A}) & \text { if $\disc V_4=1$};\\
D^{\times}(\mathbb{A}_E)& \text { if $\disc V_4\neq 1$},
\end{cases}
\]
so that we have
\[
G_3(\A)\subseteq G_4(\A)\subseteq\tilde{G}_4(\A).
\]

Now let $\pi_i$ be a tempered cuspidal automorphic representation of $G_i$ for $i=3, 4$ such that there exists a Hecke character $\omega$ with $\omega^2=\omega_{\pi_4}\omega_{\pi_3}$.

First, we assume $\disc V_4=1$. To use Ichino's  triple product formula, we need to relate $\pi_4$ with an automorphic representation of $\tilde{G}_4(\A)=D^\times(\A)\times D^\times(\A)$ as follows. By \cite[Thm. 4.13]{HS12}, there exists an irreducible cuspidal automorphic representation $\sigma_1 \boxtimes \sigma_2$ of $\tilde{G}(\mathbb{A})=D^\times(\A)\times D^\times(\A)$ on the space $V_{\sigma_1 \boxtimes \sigma_2}$ such that $V_{\pi_4} \subseteq V_{\sigma_1 \boxtimes \sigma_2}^1 |_{G_4(\A)}$ and $\sigma_1 \boxtimes \sigma_2 \|_{\G_4(\A)}=\pi_4$, where $V_{\sigma_1 \otimes \sigma_2}^1$ is the subspace of $V_{\sigma_1 \boxtimes \sigma_2}$ on which the group
$$\mathfrak{X}_{\sigma_1 \boxtimes \sigma_2}=\{ \gamma \in (Z^{\circ}_{\tilde{G}}(\mathbb{A})G_4(\mathbb{A})\tilde{G}(F)\backslash \tilde{G}(\mathbb{A}))^D:(\sigma_1\boxtimes \sigma_2) \otimes \gamma \cong \sigma_1 \boxtimes \sigma_2\}$$
acts trivially, here the superscript $D$ in the above set indicates Pontryagin dual; namely $\sigma_1\boxtimes\sigma_2$ is an automorphic representation which ``lies above $\pi_4$". (Recall that the $\|$ notation is defined in the notation section.) Moreover,  Let $\chi_1,\chi_2$ be the central characters of $\sigma_1 ,\sigma_2$, respectively.  Then $(\chi_1 \boxtimes \chi_2)|_{Z^{\circ}_4} = \chi_1\chi_2=\omega_{\pi_4}$. Since $\omega_{\pi_4}\omega_{\pi_3}=\omega^2$, we have
$\chi_1\chi_2\omega_{\pi_3}=\omega^2$.

Now, let $\phi\in V_{\pi_4}$ and $f\in V_{\pi_3}$. Since $V_{\pi_4}\subseteq V_{\sigma_1\boxtimes\sigma_2}|_{G_4(\A)}$, we may assume $\phi=(\phi_1\otimes\phi_2)|_{G_4(\A)}$ for some $\phi_i\in V_{\sigma_i}$. Then our period integral is of the form
\begin{align*}
\mathcal{P}_{\omega}(\phi, f)
&=\int_{Z_3^\circ(\A) G_3(F)\backslash G_3(\A)}\phi(g)f(g)\omega(N(g))^{-1}\,dg\\
&=\int_{Z_3^\circ(\A) G_3(F)\backslash G_3(\A)}\phi_1(g)\phi_2(g)(f\cdot\omega^{-1})(g)\,dg
\end{align*}
because $\phi(g)=\phi_1(g)\phi_2(g)$ for $g\in G_3(\A)$, which is precisely the triple product integral considered by Ichino for the automorphic representation $\sigma_1 \boxtimes \sigma_2 \boxtimes (\pi_3 \otimes \omega^{-1})$ of $D^\times(\A)\times D^\times(\A)\times D^\times(\A)$.  

However, the local integral that Ichino considers is different from our local integral.  For $\phi_{1,v},\phi_{2,v}$ in $\pi_{4,v}$ and $f_{1,v},f_{2,v}$ in $\pi_{3,v}$ our local integral is of the form 
\begin{align*}
&\alpha^{\natural}_{\omega_v}(\phi_{1,v},\phi_{2,v};f_{1,v},f_{2,v})\\
&=\int\limits_{Z_3^\circ(F_v)\backslash\G_3(F_v)}\mathcal{B}_{\pi_{4,v}}(\pi_{4,v}(g_{v})\phi_{1,v},\phi_{2,v})
\mathcal{B}_{\pi_{3,v}}(\pi_{3,v}(g_{v})f_{1,v},f_{2,v})\omega_v(N(g_v))^{-1}dg_{v}\\
&=
\int\limits_{Z_3^\circ(F_v)\backslash\G_3(F_v)}\mathcal{B}_{\pi_{4,v}}((\sigma_1 \otimes \sigma_2)_v(g_{v})\phi_{1,v},\phi_{2,v})
\mathcal{B}_{\pi_{3,v}}((\pi_{3,v}\cdot \omega_v^{-1})(g_{v})f_{1,v},f_{2,v})dg_{v}\\
&
\end{align*} because for $\sigma_1 \boxtimes \sigma_2 \|_{\G_4(\A)}=\pi_4$ we have that $\pi_{4,v} \subseteq (\sigma_{1}\otimes \sigma_{2})_v \|_{\G_{4}(F_v)}$. 

On the other hand, the local integral considered by Ichino is \begin{align*}
I_v(\phi_{1,v},\phi_{2,v};f_{1,v},f_{2,v})&=\int\limits_{Z_3^\circ(F_v)\backslash\G_3(F_v)}\mathcal{B}_{(\sigma_1 \otimes \sigma_2)_v}((\sigma_1 \otimes \sigma_2)_v(g_{v})\phi_{1,v},\phi_{2,v})
\mathcal{B}_{\pi_{3,v}}(\pi_{3,v}(g_{v})f_{1,v},f_{2,v})dg_v
\end{align*}
so we need to relate $\mathcal{B}_{(\sigma_1 \otimes \sigma_2)_v}$ with $\mathcal{B}_{\pi_{4,v}}$
Since $\pi_{4,v} \subseteq (\sigma_1 \otimes \sigma_2)_v$ for all $v$ we choose
$\mathcal{B}_{\pi_{4,v}}=\mathcal{B}_{(\sigma_1 \otimes \sigma_2)_v}|_{\pi_{4,v} \times \pi_{4,v}}$ for all $v$ except one.  We pick one place $v_o$ and set $\mathcal{B}_{\pi_{4,v_0}}=\mathcal{B}_{(\sigma_1 \otimes \sigma_2)_v}|_{\pi_{4,v} \times \pi_{4,v}}\cdot C$ for some constant $C$ which gives $\displaystyle\prod_v \mathcal{B}_{\pi_{4,v}}= \displaystyle\prod_v \mathcal{B}_{(\sigma_1 \otimes \sigma_2)_v} \cdot C $ so that $\mathcal{B}_{\pi_4}=C \cdot \mathcal{B}_{(\sigma_1 \otimes \sigma_2)}$. This constant $C$ which relates $\mathcal{B}_{(\sigma_1 \otimes \sigma_2)}$ with $\mathcal{B}_{\pi_4}$is given by Hirago-Sato in  \cite[Remark 4.20]{{HS12}}; namely,
\[
C=|\mathfrak{X}_{\sigma_1 \otimes \sigma_2}| \dfrac{Vol(Z^{\circ}_{4}(\mathbb{A})\G_4(F)\backslash \G_4 (\mathbb{A}))}{Vol(Z^{\circ}_{\tilde{G}}(\mathbb{A})\tilde{G}(F)\backslash \tilde{G}(\mathbb{A}))}.
\] 
Noting that $Z^{\circ}_{4}(\mathbb{A}){\G}_4(F)\backslash {\G}_4 (\mathbb{A})\cong \SO(F)\backslash{\SO}_4(\A)$ and $Z^{\circ}_{\tilde{G}}(\mathbb{A})\tilde{G}(F)\backslash \tilde{G}(\mathbb{A})\cong\SO_3(F)\backslash{\SO}_3(\A) \times {\SO}_3(F)\backslash\SO_3(\A)$, we have
\[
Vol(Z^{\circ}_{{\G}_4}(\mathbb{A}){\G}_4(F)\backslash {\G}_4 (\mathbb{A}))=2
\qand
Vol(Z^{\circ}_{\tilde{G}}(\mathbb{A})\tilde{G}(F)\backslash \tilde{G}(\mathbb{A}))=4,
\]
so that $C=\dfrac{1}{2}|\mathfrak{X}_{\sigma_1 \otimes \sigma_2}|$ and normalizing  $\alpha_{\omega_v}^{\sharp}$ as in the introduction gives
\[\displaystyle\prod_v \alpha_{\omega_v}=\dfrac{2}{|\mathfrak{X}_{\sigma_1\otimes\sigma_2}|} \cdot \displaystyle\prod_v I_v.\]
Thus,  Ichino's triple product formula (\cite[Theorem 1.1]{Ich08}) applied to  $\sigma_1 \boxtimes \sigma_2 \boxtimes (\pi_3 \otimes \omega^{-1})$ with $c=3$ gives
\begin{align*}
|\mathcal{P}_{\omega}(\phi, f)|^2 &=\dfrac{2}{|\mathfrak{X}_{\sigma_1 \otimes \sigma_2}|}\cdot \dfrac{\zeta(2)L(1/2, (\sigma_1 \otimes \sigma_2) \boxtimes (\pi_3 \otimes \omega^{-1}))}{8 L(1,\sigma_1 \otimes \sigma_2,\Ad)L(1,\pi_3,Ad)}\prod_v \alpha_{\omega_v}(\phi_v,f_v)\\
&=\dfrac{\Delta_{G_4}L(1/2, (\sigma_1 \otimes \sigma_2) \boxtimes (\pi_3 \otimes \omega^{-1}))}{4|\mathfrak{X}_{\sigma_1 \otimes \sigma_2}| L(1,\sigma_1 \otimes \sigma_2,\Ad)L(1,\pi_3,Ad)}\prod_v \alpha_{\omega_v}(\phi_v,f_v)\\
&=\dfrac{\Delta_{G_4}L(1/2, \pi_4 \times \pi_3 \otimes \omega^{-1})}{ 4|\mathfrak{X}_{\sigma_1 \otimes \sigma_2}|L(1,\pi_4 , \Ad)L(1,\pi_3,  \Ad)}\displaystyle\prod_v \alpha_{\omega_v}(\phi_v,f_v).
\end{align*}

Also we have that $|\mathcal{S}_{\phi_3}| = 1$ and $|\mathcal{S}_{\phi_4}| = |\mathfrak{X}_{\sigma_1 \otimes \sigma_2}|$ by \cite[Section 6.3]{LM15}. Hence, Conjectures \ref{C:conj} and \ref{C:conj2} hold.

Next assume $\disc V_4\neq 1$. Using  \cite[Thm. 4.13]{{HS12}} again, there exists an irreducible cuspidal automorphic representation $\tau$ of $\tilde{G}(\mathbb{A})=\GL_2(\A_E)$ on the space $V_{\tau}$ such that $V_{\pi_4} \subseteq V_{\tau}^1 |_{G_4}$ and $\tau \|_{\G_4}=\pi_4$; namely $\tau$ ``lies above $\pi_4$. Moreover, $V_{\tau}^1$ is the subspace of $V_{\tau}$ on which the group
$$\mathfrak{X}_{\tau}=\{ \gamma \in (Z^{\circ}_{\tilde{G}}(\mathbb{A})G_4(\mathbb{A})\tilde{G}(F)\backslash \tilde{G}(\mathbb{A}))^D:\tau \otimes \gamma \cong \tau\}$$
acts trivially. Note that $\omega_\tau|_{Z^\circ_4(\A)}=\omega_{\pi_4}$, and since $\omega_{\pi_4}\omega_{\pi_3}=\omega^2$ we have $\omega_\tau|_{Z^\circ_4(\A)}\omega_{\pi_3}=\omega^2$. Now let $\phi\in V_{\pi_4}$ and $f\in V_{\pi_3}$. Since $V_{\pi_4}\subseteq V_\tau|_{G_4(\A)}$, we can write $\phi=\tilde{\phi}|_{G_4(\A)}$ for some $\tilde{\phi}\in V_{\tau}$. Then our period integral is of the form
\begin{align*}
\mathcal{P}_{\omega}(\phi, f)
&=\int_{Z_4^\circ(\A) G_3(F)\backslash G_3(\A)}\phi(g)f(g)\omega(N(g))^{-1}\,dg\\
&=\int_{Z_4^\circ(\A) G_3(F)\backslash G_3(\A)}\tilde{\phi}(g)(f\cdot\omega^{-1})(g)\,dg
\end{align*}
because $\phi(g)=\tilde{\phi}(g)$ for $g\in G_3(\A)$, which is precisely the period integral considered by Ichino for the automorphic representation $\tau\boxtimes(\pi_3\otimes\omega^{-1})$ of $D^\times(\A_E)\times D^\times(\A)$. 

However, again the local integral that Ichino considers is different than our local integral.  For $\phi_{1,v},\phi_{2,v}$ in $\pi_{4,v}$ and $f_{1,v},f_{2,v}$ in $\pi_{3,v}$ our local integral is of the form 
\begin{align*}
&\alpha^{\natural}_{\omega_v}(\phi_{1,v},\phi_{2,v};f_{1,v},f_{2,v})\\
&=\int\limits_{Z_3^\circ(F_v)\backslash\G_3(F_v)}\mathcal{B}_{\pi_{4,v}}(\pi_{4,v}(g_{v})\phi_{1,v},\phi_{2,v})
\mathcal{B}_{\pi_{3,v}}(\pi_{3,v}(g_{v})f_{1,v},f_{2,v})\omega_v(N(g_v))^{-1}dg_{v}\\
&=
\int\limits_{Z_3^\circ(F_v)\backslash\G_3(F_v)}\mathcal{B}_{\pi_{4,v}}(\tau_v(g_{v})\phi_{1,v},\phi_{2,v})
\mathcal{B}_{\pi_{3,v}}((\pi_{3,v}\cdot \omega_v^{-1})(g_{v})f_{1,v},f_{2,v})dg_{v}\\
&
\end{align*} because for $\tau \|_{\G_4(\A)}=\pi_4$ we have that $\pi_{4,v} \subseteq \tau_v \|_{\G_{4}(F_v)}$. 

On the other hand, the local integral considered by Ichino is \begin{align*}
I_v(\phi_{1,v},\phi_{2,v};f_{1,v},f_{2,v})&=\int\limits_{Z_3^\circ(F_v)\backslash\G_3(F_v)}\mathcal{B}_{\tau_v}(\tau_v(g_{v})\phi_{1,v},\phi_{2,v})
\mathcal{B}_{\pi_{3,v}}(\pi_{3,v}(g_{v})f_{1,v},f_{2,v})dg_v,
\end{align*}
so we need to relate $\mathcal{B}_{\tau_v}$ with $\mathcal{B}_{\pi_{4,v}}$. 
Since $\pi_{4,v} \subseteq \tau_v$ for all $v$ we choose
$\mathcal{B}_{\pi_{4,v}}=\mathcal{B}_{\tau_v}|_{\pi_{4,v} \times \pi_{4,v}}$ for all $v$ except one and pick one place $v_o$ and set $\mathcal{B}_{\pi_{4,v_0}}=\mathcal{B}_{\tau_v}|_{\pi_{4,v} \times \pi_{4,v}}\cdot C'$ for some constant $C'$ which gives $\displaystyle\prod_v \mathcal{B}_{\pi_{4,v}}= \displaystyle\prod_v \mathcal{B}_{\tau_v} \cdot C'$ and so $\mathcal{B}_{\pi_4}=C'\cdot \mathcal{B}_{\tau}$ where $C'$ is given by Hirago-Sato in  \cite[Remark 4.20]{{HS12}};namely,
\[
C'=|\mathfrak{X}_{\tau}| \dfrac{Vol(Z^{\circ}_{4}(\mathbb{A})\G_4(F)\backslash \G_4 (\mathbb{A}))}{Vol(Z^{\circ}_{\tilde{G}}(\mathbb{A})\tilde{G}(F)\backslash \tilde{G}(\mathbb{A}))}
\] 
where the volumes are 
$$Vol(Z^{\circ}_{G_4}(\mathbb{A})G_4(F)\backslash G_4 (\mathbb{A}_F))=Vol(Z^{\circ}_{\tilde{G}}(\mathbb{A})\tilde{G}(F)\backslash \tilde{G}(\mathbb{A_F}))=2,$$
and so $C'=|\mathfrak{X}_{\tau}|$ in this case.

Then Ichino's triple product formula applied to $\tau\boxtimes(\pi_3\otimes\omega^{-1})$ with $c=2$ gives
\begin{align*}
|\mathcal{P}_\omega(\phi,f)|^2&=\dfrac{1}{\mathfrak{X}_{\tau}}\cdot\dfrac{\zeta(2)L(1/2, \tau \boxtimes (\pi_3 \otimes \omega^{-1}))}{4 L(1,\tau,\Ad)L(1,\pi_3,Ad)}\displaystyle\prod_v \alpha_{\omega_v}(\phi_v,f_v)\\
&=\dfrac{\Delta_{G_4}L(1/2, \tau \boxtimes (\pi_3 \otimes \omega^{-1}))}{4|\mathfrak{X}_{\tau}| L(1,\tau,\Ad)L(1,\pi_3,Ad)}\displaystyle\prod_v \alpha_{\omega_v}(\phi_v,f_v)\\
&=\dfrac{\Delta_{G_4}L(1/2, \pi_4 \times \pi_3 \otimes \omega^{-1})}{ 4|\mathfrak{X}_{\tau}|L(1,\pi_4 , \Ad)L(1,\pi_3,  \Ad)}\displaystyle\prod_v \alpha_{\omega_v}(\phi_v,f_v).
\end{align*}
Moreover, $|\mathcal{S}_{\phi_3}|=1$ and $|\mathcal{S}_{\phi_4}|=|\mathfrak{X}_{\tau}|$ by \cite[Section 6.4]{LM15}.  Hence, Conjectures \ref{C:conj} and \ref{C:conj2} hold for this case, too.

\begin{remark}
In Ichino's triple product formula, the two cases $d=1$ and $d\neq 1$ are different. However, once we consider Ichino's formula as an instance of the GGP conjecture for $\GSpin$, these two cases can be considered as one case as above.
\end{remark}

\subsection{Conjecture for $\GSpin_5 \times \GSpin_4$ (Gan-Ichino formula)}

Finally, we consider the conjecture for $n=5$. For this case, we will not be able to prove the conjecture in full generality but only for some special cases as we will explain in what follows.

Firstly we consider only the case where $\GSpin_5$ is split and $\GSpin_4$ is quasi-split, namely
\begin{align*}
\GSpin_4&=
\begin{cases}
\{(g_1,g_2)\in\GL_1\times\GL_1 : \det g_1=\det g_2\};\\
\{g\in\Res_{E/F}\GL_2 : N(g)\in\GL_1\},
\end{cases}\\
\GSpin_5&=\GSp_4.
\end{align*}
Or equivalently, $\GSpin_4=\GSpin(V_4)$ with $V_4=\mathbb{H}\oplus \mathbb{H}$ or $V_4=\mathbb{H}\oplus E$ for a quadratic extension $E$ of $F$ equipped with the norm form, and $\GSpin_5=\GSpin(V_5)$ with $\mathbb{H} \oplus \mathbb{H} \oplus \langle 1 \rangle$. Let us note that we have the natural inclusion $V_4\subseteq V_5$ of the quadratic forms, which gives rise to the natural inclusion $\GSpin_4\subseteq\GSpin_5$.

Secondly, for our (tempered) cuspidal automorphic representations $\pi_5$ and $\pi_4$ of $\GSpin_5(\A)$ and $\GSpin_4(\A)$, respectively, we only consider the following special cases. For $\pi_5$, we assume that $\pi_5$ is the theta lift of a cuspidal automorphic representation $\sigma$ of $\GO(V)(\A)$, where $V$ is a 4 dimensional quadratic space. (Let us note that $\sigma$ has to satisfy certain technical conditions such that the theta lift to $\GSp_4(\A)$ is nonzero and cuspidal. See \cite[p.236]{GI11} for the detail.) Here, let us denote the discriminant algebra of $V$ by $K$, which is the \'etale quadratic algebra over $F$ defined by
\[
K=\begin{cases} F \times F & \text {if $\disc(V)=1$};\\
 F(\sqrt{\disc(V)})& \text {if $\disc(V)\neq1$}.
 \end{cases}
\]
As for $\pi_4$, let $\tau$ be a cuspidal automorphic representation which ``lies above $\pi_4$" as in the previous subsection. To be precise, let
\[
\tilde{G}_4:=
\begin{cases}
\GL_2\times\GL_2,&\text{if $V_4=\mathbb{H}\oplus\mathbb{H}$};\\
\Res_{E/F}\GL_2,&\text{if $V_4=\mathbb{H}\oplus E$}.
\end{cases}
\]
By Theorem 4.13 of \cite{{HS12}}, there exists an irreducible unitary cuspidal automorphic representation $\tau$ of $\widetilde{G_4}(\A)$ such that $\tau\|_{\G_4}=\pi_4$ and $V_{\pi_4} \subset V^1_{\tau}|_{G_4}$, where $V^1_{\tau}$ is the subspace of $V_{\tau}$ such that
$$\mathfrak{X}_{\tau} = \{\gamma \in (Z^{\circ}_{\widetilde{G}}(\mathbb{A})G_4(\mathbb{A})\widetilde{G_4}(F) \backslash \widetilde {G_4}(\mathbb{A})^D : \tau \otimes \gamma \cong \tau \}$$
acts trivially.  Then we assume that the base change $\tau_K$ of $\tau$ to $\widetilde{G}(\mathbb{A}_K)$ is cuspidal, where $K$ is the discriminant algebra of $V$ as above, and the Jacquet-Langlands transfer of $\tau_K$ to $D^{\times}(\mathbb{A}_{E \otimes K})$ exists. (See the top of \cite[pg.\ 237]{GI11} for the detail.)

Then Gan-Ichino essentially proved the following.
\begin{theorem}[Gan-Ichino]
Let $\pi_5$ and $\pi_4$ be as above. Assume there exists a Hecke character $\omega$ such that $\omega^2=\omega_{\pi_5}\omega_{\pi_4}$. Then for factorizable $\phi\in V_{\pi_5}$ and $f\in V_{\pi_4}$, we have
\[
|\mathcal{P}_\omega(\phi, f)|^2=\frac{\zeta(2)\zeta(4)L(1/2,\pi_5 \times \pi_4\otimes \omega^{-1})}{2^{\beta}|\mathfrak{X}_{\tau}|L(1,\pi_5,\Ad)(L(1,\pi_4,\Ad) }\prod_v \alpha_{\omega_v}(\phi_v, f_v),
\]
where
\[
\beta=
\begin{cases} 3 & \text { if $\disc (V)=1$};\\
2 & \text { if $\disc(V)\neq 1$}.
\end{cases}
\]
\end{theorem}
\begin{proof}
This is essentially Theorem 1.1 of \cite{GI11} with the notation adjusted to ours by setting $\pi=\pi_5$ and $\pi'=\pi_4\otimes\omega^{-1}$, where $\pi$ and $\pi'$ are as in \cite{GI11}. But it should be mentioned that if $\pi_4$ satisfies the above mentioned conditions, then so does $\pi_4\otimes\omega^{-1}$, and hence we can use the Gan-Ichino formula for $\pi_5\times\pi_4\otimes\omega^{-1}$.
\end{proof}

Now, let us take care of the constant $2^{\beta}|\mathfrak{X}_{\tau}|$. By \cite[Section 6]{LM15}, we know $|\mathfrak{X}_{\tau}|=|\mathcal{S}_{\phi_4}|$.  For the representation $\pi_5$ considered here, Roberts essentially verified in \cite{Rob01} that
$$|\mathcal{S}_{\phi_5}|=\begin{cases} 2 \text { if $\disc(V)=1$};\\
1 \text{ if $\disc(V)\neq 1$}.\\
\end{cases}
$$
Hence if $\disc(V)=1$ we have
\[
2^{\beta}|\mathfrak{X}_{\tau}|=2^3|\mathfrak{X}_{\tau}|=2^2|\mathcal{S}_{\phi_5}||\mathcal{S}_{\phi_4}|,
\]
and if $\disc V\neq 1$ we have
\[
2^{\beta}|\mathfrak{X}_{\tau}|=2^2|\mathfrak{X}_{\tau}|=2^2|\mathcal{S}_{\phi_5}||\mathcal{S}_{\phi_4}|.
\]
Thus in either case the above theorem confirms Conjectures \ref{C:conj} and \ref{C:conj2}.

\begin{remark}
In the above theorem, Gan-Ichino assumed that $F$ and $E$ are totally real number fields. This assumption was to utilize the Siegel-Weil formula. (See \cite[Remark 1.3]{GI11}.) However, the condition is no longer necessary thanks to the work of Gan-Qui-Takeda in \cite{GTQ14}. Also in \cite{Rob01}, Roberts assumed that $F$ and $E$ are totally real essentially for the same reason, and hence this assumption is not necessary, either.

Also Gan-Ichino do not assume that $\pi_5$ and $\pi_4$ are tempered. This is because for the case at hand the convergence of the local integral as we did in Proposition \ref{P:convergence_local_period} can be shown by using the Kim-Shahidi estimate as in \cite[Lemma 9.1]{GI11}. Hence we do not even need to assume $\pi_5$ and $\pi_4$ are tempered.
\end{remark}

\bibliographystyle{abbrv}
\addcontentsline{toc}{section}{Bibliography}
\bibliography{globalgspin}

\end{document}